\documentclass[a4paper,oneside,reqno]{amsart}
\usepackage[utf8]{inputenc}
\usepackage[a4paper,margin=3cm]{geometry} 
\usepackage{bm}
\usepackage{amsmath}
\usepackage{amsthm}
\usepackage{amscd}
\usepackage{amssymb}
\usepackage{MnSymbol}
\usepackage{latexsym}
\usepackage{eucal}
\usepackage{dsfont}
\usepackage{mathtools}
\usepackage{enumitem}

\usepackage{graphicx}
\usepackage{float}
\usepackage{array,booktabs}

\usepackage[colorlinks,pdftex]{hyperref}
\hypersetup{
linkcolor=black,
citecolor=black,
pdftitle={},
pdfauthor={Franziska K\"uhn},
pdfkeywords={},
}

\renewcommand\labelenumi{\textup{(\roman{enumi})}}
\renewcommand\theenumi\labelenumi
\renewcommand\labelenumii{(\alph{enumii})}
\renewcommand\theenumii\labelenumii

\renewcommand\theenumii\labelenumii

\swapnumbers
\theoremstyle{theorem} \newtheorem{theorem}{Theorem}[section]
\theoremstyle{theorem} \newtheorem{lemma}[theorem]{Lemma}
\theoremstyle{theorem} \newtheorem{proposition}[theorem]{Proposition}
\theoremstyle{theorem} \newtheorem{corollary}[theorem]{Corollary}
\theoremstyle{definition} 
\theoremstyle{remark} \newtheorem{remark}[theorem]{Remark}
\theoremstyle{definition}  \newtheorem{example}[theorem]{Example}

\DeclareMathOperator \re {Re}
\DeclareMathOperator \im {Im}

\DeclareMathOperator \essinf {essinf}

\DeclareMathOperator \tr {tr}

\newcommand{\I}{\mathds{1}}

\newcommand\fa{\qquad \text{for all \ }}

\newcommand{\cadlag}{c\`adl\`ag }

\newcommand\mc[1] {\mathcal{#1}}
\newcommand\mbb[1] {\mathds{#1}}

\newcommand{\eps}{\varepsilon}

\newcommand{\primo}{1\textsuperscript{o}}
\newcommand{\primolist}{\par\smallskip\textbf\primo\;\;}
\newcommand{\secundo}{2\textsuperscript{o}}
\newcommand{\secundolist}{\par\smallskip\textbf\secundo\;\;}
\newcommand{\tertio}{3\textsuperscript{o}}
\newcommand{\tertiolist}{\par\smallskip\textbf\tertio\;\;}
\newcommand{\quarto}{4\textsuperscript{o}}
\newcommand{\quartolist}{\par\smallskip\textbf\quarto\;\;}

\makeatletter
\@namedef{subjclassname@2020}{%
  \textup{2020} Mathematics Subject Classification}
\makeatother

\begin{document}

\title[{Upper functions for L\'evy(-type) processes}]{Upper functions for sample paths of L\'evy(-type) processes}
\author[F.~K\"{u}hn]{Franziska K\"{u}hn} 
\address[F.~K\"{u}hn]{TU Dresden, Fakult\"at Mathematik, Institut f\"{u}r Mathematische Stochastik, 01062 Dresden, Germany. E-Mail: \textnormal{franziska.kuehn1@tu-dresden.de}}


\begin{abstract}
	We study the small-time asymptotics of sample paths of L\'evy processes and L\'evy-type processes. Namely, we investigate under which conditions the limit
	$$
		\limsup_{t \to 0} \frac{1}{f(t)} |X_t-X_0|
	$$ 
	is finite resp.\ infinite with probability $1$. We establish integral criteria in terms of the infinitesimal characteristics and the symbol of the process. Our results apply to a wide class of processes, including solutions to L\'evy-driven SDEs and stable-like processes. For the particular case of L\'evy processes, we recover and extend earlier results from the literature. Moreover, we present a new maximal inequality for L\'evy-type processes, which is of independent interest.
\end{abstract}
\maketitle

\section{Introduction} \label{intro}

A mapping $f:[0,1] \to [0,\infty)$ is called an upper function for a stochastic process $(X_t)_{t \geq 0}$ if
\begin{equation*}
	\limsup_{t \to 0} \frac{1}{f(t)} |X_t-X_0| \leq 1 \quad \text{almost surely},
\end{equation*}
i.e.\ a typical sample path $t \mapsto X_t(\omega)$ grows asymptotically at most as fast as $f(t)$. In this article, we are interested in upper functions for L\'evy and L\'evy-type processes. Our aim is to establish integral criteria in terms of the characteristics and the symbol of the process -- see Section~\ref{def} for definitions -- which characterize whether $f$ is an upper function.  \par
For L\'evy processes, the study of upper functions was initiated by Khintchine \cite{khin39}. He showed that any one-dimensional L\'evy process satisfies the following law of iterated logarithm (LIL):
\begin{equation*}
	-  \liminf_{t \to 0} \frac{X_t}{\sqrt{2t \log \log \frac{1}{t}}} 
	= \limsup_{t \to 0} \frac{X_t}{\sqrt{2 t \log \log \frac{1}{t}}}
	= \sigma \quad \text{a.s.},
\end{equation*}
where $\sigma \geq 0$ is the diffusion coefficient. In consequence, the small-time asymptotics of a L\'evy process is governed by the Gaussian part if $\sigma \neq 0$. For this reason our focus is on processes with vanishing diffusion part. Khintchine \cite{khin39} also showed -- under some mild assumptions -- that $f$ is an upper function for a L\'evy process $(X_t)_{t \geq 0}$ if
\begin{equation*}
	\int_0^1 \frac{1}{t} \mbb{P}(|X_t| \geq c f(t)) \, dt < \infty
\end{equation*}
for a suitable constant $c>0$. In practice, it is often difficult to check whether the latter integral is finite. There is a more tractable criterion in terms of the L\'evy measure $\nu$. Namely, it holds for a wide class of functions $f$ that
\begin{equation}
	\limsup_{t \to 0} \frac{1}{f(t)} |X_t| = \begin{rcases}
	    \begin{dcases}  0 \\  \infty \end{dcases} \end{rcases}\, \, \text{a.s.} 
	\iff
	\int_{0}^1 \nu(\{y \in \mbb{R}^d; |y| \geq f(t)\}) \, dt \,\, \begin{rcases}
	    \begin{dcases}< \infty \\ = \infty \end{dcases} \end{rcases};
	\label{intro-eq1}
\end{equation}
this characterization is classical for stable L\'evy processes, see e.g.\ \cite{fristedt}, and has been extended to general one-dimensional L\'evy processes by Wee \& Kim \cite{wee88}. For some processes, \eqref{intro-eq1} breaks down, and it may happen that $\limsup_{t \to 0} \frac{1}{f(t)} |X_t| \in (0,\infty)$ almost surely, see \cite{bertoin08,savov09,wee88} for details. A number of further characterizations for upper functions are collected in Theorem~\ref{main-3}. We require only mild assumptions on the L\'evy process $(X_t)_{t \geq 0}$ and the mapping $f$; thus generalizing earlier results in the literature. For power functions $f(t)=t^{\kappa}$, there is a close connection to the Blumenthal--Getoor index, cf.\ Corollary~\ref{main-47}. \par
The second part of our results is about the small-time asymptotics of L\'evy-type processes. Intuitively, a L\'evy-type process behaves locally like a L\'evy process but the L\'evy triplet depends on the current position of the process, see Section~\ref{def} below for the precise definition. Important examples include solutions to L\'evy-driven stochastic differential equations (SDEs), processes of variable order and random time changes of L\'evy processes, just to mention a few. Studying the sample path behaviour of L\'evy-type processes is much more delicate than in the L\'evy case because the processes are no longer homogeneous in space, see \cite[Chapter 5]{ltp} for a survey on results for the closely related class of Feller processes. Schilling \cite{rs-growth} introduced generalized Blumenthal--Getoor index and obtained a criterion for a power function $f(t)=t^{\kappa}$ to be an upper function of a L\'evy-type process, see also \cite[Theorem 5.16]{ltp}. A recent paper by Reker \cite{reker20} studies the small-time asymptotics of solutions to SDEs driven by jump processes. Moreover, there are LIL-type results for L\'evy-type processes and other classes of jump processes available in the literature, see \cite{kim21,kim17,knop14} and the references therein. Our contribution in this paper is two-fold. Firstly, we establish sufficient conditions in terms of the characteristics and the symbol, which ensure that a given mapping $f$ is an upper function for a L\'evy-type process, cf.\ Theorem~\ref{main-5}. On the way, we obtain new results on upper functions for Markov processes, cf.\ Section~\ref{up}. Secondly, we prove a criterion for a given function $f$ \emph{not} to be an upper function, cf.\ Theorem~\ref{main-9}. The key ingredients for the proofs are a new maximal inequality for L\'evy-type processes, cf.\ Section~\ref{max}, and a conditional Borel--Cantelli lemma for backward filtrations.

\section{Main results} \label{main}

This section is divided into two parts: First, we present our results for L\'evy processes and then, in the second part, we state the results which apply for the wider class of L\'evy-type processes. See Section~\ref{def} below for definitions and notation. The following is our first main result.

\begin{theorem} \label{main-3}
	Let $(X_t)_{t \geq 0}$ be a L\'evy process with L\'evy triplet $(b,0,\nu)$ and characteristic exponent $\psi$ satisfying the sector condition, i.e.\ $|\im \psi(\xi)| \leq C \re \psi(\xi)$, $\xi \in \mbb{R}^d$, for some constant $C>0$. Let $f: [0,1] \to (0,\infty)$ be non-decreasing, and assume that one of the following conditions holds.
	\begin{enumerate}[label*=\upshape (A\arabic*),ref=\upshape A\arabic*] 
		\item\label{A1} The L\'evy measure $\nu$ satisfies \begin{equation*}
			\limsup_{r \to 0} \frac{\int_{|y| \leq r} |y|^2 \, \nu(dy)}{r^2 \nu(\{|y| > r\})} < \infty. 
		\end{equation*}
		\item\label{A2} There is a constant $M>0$ such that \begin{equation*}
				\int_{r<f(t)} \frac{1}{f(t)^2} \, dt \leq M \frac{f^{-1}(r)}{r^2}, \qquad r \in (0,1).
			\end{equation*}
	\end{enumerate}
	The following statements are equivalent. 
	\begin{enumerate}[label*=\upshape (L\arabic*),ref=\text{\upshape (L\arabic*)}]
		\item\label{main-3-i} $\int_0^1 \nu(\{|y| \geq cf(t)\}) \, dt < \infty$ for some $c>0$,
		\item\label{main-3-ii} $\int_0^1 \sup_{|\xi| \leq 1/(\eps f(t))} |\psi(\xi)| \, dt<\infty$ for some (all) $\eps>0$,
		\item\label{main-3-iii} $\int_0^1 \mbb{P} \left( \sup_{s \leq t} |X_s| \geq \eps f(t)\right) \frac{1}{t} \, dt < \infty$ for all $\eps>0$,
		\item\label{main-3-iv} $\int_0^1 \mbb{P}(|X_t| \geq \eps f(t)) \frac{1}{t} \, dt < \infty$ for all $\eps>0$,
		\item\label{main-3-v} $\limsup_{t \to 0} \frac{1}{f(t)} \sup_{s \leq t} |X_s|=0$ almost surely,
		\item\label{main-3-vi} $\limsup_{t \to 0} \frac{1}{f(t)} |X_t|=0$ almost surely,
		\item\label{main-3-vii} $\limsup_{t \to 0} \frac{1}{f(t)} |X_t|<\infty$ almost surely.
	\end{enumerate}
	In particular,
	\begin{equation*}
		\limsup_{t \to 0} \frac{1}{f(t)} |X_t| \in \{0,\infty\} \quad\text{a.s.}
	\end{equation*}
	In \ref{main-3-v}-\ref{main-3-vii} we may replace 'almost surely' by 'with positive probability'.
\end{theorem}

Theorem~\ref{main-3} generalizes  \cite[Corollary 11.3]{fristedt} for stable processes and the results from \cite[Section III.4]{bertoin} for subordinators. Savov \cite{savov09} proved the equivalence of \ref{main-3-i} and \ref{main-3-vi} under (a bit stronger condition than) \eqref{A2} and the assumption that $(X_t)_{t \geq 0}$ has paths of unbounded variation. The equivalence of \ref{main-3-iv} and \ref{main-3-vi} goes back to Khintchine \cite{khin39}, see also \cite[Appendix, Theorem 2]{skorohod91}. The proof of Theorem~\ref{main-3} will be presented in Section~\ref{p}.

\begin{remark} \label{main-4} \begin{enumerate}[wide, labelwidth=!, labelindent=0pt]
	\item\label{main-4-ii} The proof of Theorem~\ref{main-3} shows that the implications \begin{equation*}
		\ref{main-3-ii}\implies\ref{main-3-iii}\implies\ref{main-3-iv}\implies\ref{main-3-v}\implies\ref{main-3-vi}\implies\ref{main-3-vii}\implies\ref{main-3-i}
	\end{equation*} 
	hold without the sector condition. The sector condition is only needed to relate the integrals in \ref{main-3-i} and \ref{main-3-ii} to each other. In fact, the key for the proof of \ref{main-3-i} $\implies$\ref{main-3-ii} is the implication
	\begin{equation}
		\exists c>0\::\: \int_0^1 \nu(\{|y| \geq cf(t)\}) \, dt < \infty
		\implies \forall \eps>0\::\: \int_0^1 \sup_{|\xi| \leq 1/(\eps f(t))} \re \psi(\xi) \, dt<\infty, \label{main-eq7}
	\end{equation}
	which does not require the sector condition, see Lemma~\ref{p-1} below; the sector condition is then used to replace $\re \psi$ by $|\psi|$ in the integral on the right-hand side. 
	\item\label{main-4-iii} For the equivalences to hold, it is crucial that one of the assumptions \eqref{A1}, \eqref{A2} is satisfied; if both assumptions are violated, then the equivalences break down in general and it may happen that
	\begin{equation*}
		0< \limsup_{t \to 0} \frac{1}{f(t)} |X_t| < \infty \quad \text{a.s.},
	\end{equation*}
	see \cite{bertoin08,savov09,wee88} and Example~\ref{main-48} below.
	\item\label{main-4-i} Condition \eqref{A2} holds for any continuous function $f:[0,1] \to [0,\infty)$ satisfying $\frac{f(t)}{t} \uparrow \infty$ as $t \downarrow 0$ and $\frac{f(t)}{t^{\alpha}} \downarrow 0$ as $t \downarrow 0$ for some $\alpha>\tfrac{1}{2}$, cf.\ \cite[Proof of Corollary 2.1]{savov09}. While this criterion is useful in many cases, it is too restrictive in some situations. For instance, if $(X_t)_{t \geq 0}$ is an isotropic $\alpha$-stable L\'evy process, then $f(t)=t^{1/(\alpha-\eps)}$ is an upper function, cf.\ Example~\ref{main-45} below, but clearly  $\frac{f(t)}{t} \uparrow \infty$ as $t \to 0$ fails to hold if $\alpha<1$. On the other hand, a straightforward calculation shows that the L\'evy measure of the isotropic $\alpha$-stable L\'evy process satisfies \eqref{A1}, and therefore Theorem~\ref{main-3} applies in this case without any additional growth assumptions on $f$. For further comments on \eqref{A1} and equivalent formulations, we refer to Remark~\ref{p-3}.
\end{enumerate} \end{remark}

Let us illustrate Theorem~\ref{main-3} with an example.

\begin{example} \label{main-45} 
	Let $(X_t)_{t \geq 0}$ be an $\alpha$-stable pure-jump L\'evy process, $\alpha \in (0,2)$, that is, a L\'evy process with L\'evy triplet $(0,0,\nu)$ where the L\'evy measure $\nu$ is of the form
	\begin{equation*}
		\nu(A) = \int_0^{\infty} \int_{\mbb{S}^{d-1}} \I_A(r \theta) \frac{1}{r^{1+\alpha}} \, \mu(d\theta) \, dr
	\end{equation*}
	for a measure $\mu$ on the sphere $\mbb{S}^{d-1}$ satisfying $\mu(\mbb{S}^{d-1})>0$, see \cite{sato} for a thorough discussion of stable processes. Theorem~\ref{main-3} shows that
	\begin{equation}
		\limsup_{t \to 0} \frac{1}{f(t)} |X_t|= \begin{rcases}
			    \begin{dcases} 0 \\ \infty \end{dcases} \end{rcases} \, \, \text{a.s.} 
		\iff \int_0^1 |f(t)|^{-\alpha} \, dt \,\, \begin{rcases}
	    \begin{dcases} < \infty \\ = \infty \end{dcases} \end{rcases} 
		\label{main-eq9}
	\end{equation}
	for any non-decreasing function $f:[0,1] \to [0,\infty)$, and so we recover the classical characterization for upper functions of sample paths of stable L\'evy processes, see e.g.\ \cite[Corollary 11.3]{fristedt}.
\end{example}

For power functions $f(t)=t^{\kappa}$, the finiteness of $\limsup_{t \to 0} \frac{1}{f(t)} |X_t|$ can be characterized in terms of the Blumenthal--Getoor index \begin{equation*}
	\beta := \inf\left\{ \alpha>0; \int_{|y|<1} |y|^{\alpha} \, \nu(dy)< \infty \right\} \in [0,2],
\end{equation*}
which was introduced in \cite{bg61}. The following result is immediate from Theorem~\ref{main-3}.

\begin{corollary} \label{main-47}
	Let $(X_t)_{t \geq 0}$ be a L\'evy process with L\'evy triplet $(b,0,\nu)$ and assume that the characteristic exponent satisfies the sector condition. Then
	\begin{equation}
		\limsup_{t \to 0} \frac{1}{t^{\kappa}} |X_t| = \begin{rcases}
	    \begin{dcases}0 \\ \infty \end{dcases} \end{rcases}\, \, \text{a.s.} 
		\iff \int_{|y|<1} |y|^{1/\kappa} \, \nu(dy) \begin{rcases}
	    \begin{dcases}<\infty \\ = \infty \end{dcases} \end{rcases}
		\label{main-eq4}
	\end{equation}
	for every $\kappa \in (\tfrac{1}{2},\infty)$, and
	\begin{equation}
		\limsup_{t \to 0} \frac{1}{t^{\kappa}} |X_t| = \begin{rcases}
	    \begin{dcases}0 \\ \infty \end{dcases} \end{rcases} \, \, \text{a.s.} \quad \text{according as} \quad \begin{rcases}
	    \begin{dcases}\kappa < 1/\beta \\ \kappa>1/\beta \end{dcases} \end{rcases}
		\label{main-eq5}
	\end{equation}
	for all $\kappa>0$. If \eqref{A1} from Theorem~\ref{main-3} is satisfied, then 
	\begin{equation*}
		\limsup_{t \to 0} \frac{1}{\sqrt{t}} |X_t| = 0 \quad \text{a.s.}
	\end{equation*}
\end{corollary}	

The characterization \eqref{main-eq5} goes back to Pruitt \cite{pruitt81} and Blumenthal \& Getoor \cite{bg61}, see also \cite[Proposition 47.24]{sato}. Note that the critical case $\kappa=1/\beta$ is excluded in \eqref{main-eq5}; one has to check by hand whether the integral $\int_{|y|<1}|y|^{\beta} \, \nu(dy)$ is finite. In \eqref{main-eq4} the critical case is $\kappa=\frac{1}{2}$; this is due to the fact that $\int_{|y| < 1} |y|^2 \, \nu(dy) < \infty$ is always satisfied but at the same time there are pure-jump L\'evy processes with
\begin{equation*}
	0< \limsup_{t \to 0} \frac{1}{\sqrt{t}} |X_t| < \infty \quad \text{a.s.}
\end{equation*}
cf.\ \cite{bertoin08,wee88}. Consequently, \eqref{main-eq4} fails, in general, for $\kappa=\tfrac{1}{2}$. Let us give an example of such a process and explain why this is not a contradiction to Theorem~\ref{main-3}.

\begin{example} \label{main-48}
	Let $(X_t)_{t \geq 0}$ be a one-dimensional L\'evy process with L\'evy triplet $(0,0,\nu)$ and L\'evy measure
	\begin{equation*}
		\nu(dy) = \frac{1}{2} \frac{1}{|y|^2} \varphi'(|y|) \I_{(0,1/e^e)}(|y|) \, dy
	\end{equation*}
	for $\varphi(r)=1/\log \log \frac{1}{r}$. Note that $\nu$ is indeed a L\'evy measure, i.e.\ $\int \min\{|y|^2,1\} \, \nu(dy)<\infty$. As
	\begin{equation}
		\int_{|y| \leq r} |y|^2 \, \nu(dy) = \varphi(r), \label{main-eq6}
	\end{equation}
	it follows from \cite[Theorem 2.2]{bertoin08} that 
	\begin{equation*}
		\limsup_{t \to 0} \frac{1}{\sqrt{t}} |X_t| = \sqrt{2} \quad \text{a.s.}
	\end{equation*}
	In particular, \eqref{main-eq4} breaks down for $\kappa=\tfrac{1}{2}$ and the equivalences in Theorem~\ref{main-3} fail to hold for $f(t)=\sqrt{t}$. This is not, however, a contradiction to Theorem~\ref{main-3} because the assumptions \eqref{A1} and \eqref{A2} in Theorem~\ref{main-3} are both violated. It is straightforward to check that \eqref{A2} fails for $f(t)=\sqrt{t}$; to see that \eqref{A1} fails we note that, by the Karamata's Tauberian theorem, see e.g.\ \cite{bingham},
	\begin{equation*}
		\nu(\{|y| > r\}) 
		= \int_r^{1/e^e} \frac{1}{y^2} \varphi'(y) \, dy  
		\approx \frac{1}{2} r^{-2} \frac{1}{\log \frac{1}{r} \left(\log \log \frac{1}{r}\right)^2} \quad \text{as $r \to 0$},
	\end{equation*}
	and thus, by \eqref{main-eq6} and the definition of $\varphi$,
	\begin{equation*}
		\lim_{r \to 0} \frac{\int_{|y| \leq r} |y|^2 \, \nu(dy)}{r^2 \nu(\{|y|>r\})} = \infty.
	\end{equation*}
\end{example}

Next we present our results for the wider class of L\'evy-type processes, see Section~\ref{def} below for the definition. The following theorem gives sufficient conditions for an increasing function $f:[0,1] \to [0,\infty)$ to be an upper function of a L\'evy-type process. 

\begin{theorem} \label{main-5} 
	Let $(X_t)_{t \geq 0}$ be a L\'evy-type process with characteristics $(b(x),0,\nu(x,dy))$ and symbol $q$ satisfying the sector condition. Let $x \in \mbb{R}^d$ and $R>0$ such that 
	 \begin{equation*}
				M:=\limsup_{r \to 0} \sup_{|z-x| \leq R} \frac{\int_{|y| \leq r} |y|^2 \, \nu(z,dy)}{r^2 \nu(z,\{|y| > r\})} < \infty.
				\tag{A1'}
				\label{A1'}
		\end{equation*}
	Then the implications
		\begin{equation*}
			\ref{main-5-i} \implies \ref{main-5-ii} \implies\ref{main-5-iii} \implies \ref{main-5-iv} 
		\end{equation*}
	hold for any non-decreasing function $f:[0,1] \to (0,\infty)$, where
	\begin{enumerate}[label*=\upshape (LTP\arabic*),ref=\text{\upshape (LTP\arabic*)}]
		\item\label{main-5-i} $\int_0^1 \sup_{|z-x| \leq f(t)} \nu(z,\{|y| \geq c f(t)\}) \, dt < \infty$ for some $c>0$, 
		\item\label{main-5-ii} $\int_0^1 \sup_{|z-x| \leq f(t)} \sup_{|\xi| \leq 1/(\eps f(t))} |q(z,\xi)| \, dt < \infty$ for some (all) $\eps>0$,
		\item\label{main-5-iii} $\int_0^1 \sup_{|z-x| \leq f(t)} \mbb{P}^z \left( \sup_{s \leq t} |X_s-z| \geq \eps f(t) \right) \, \frac{1}{t} \, dt < \infty$ for all $\eps>0$,
		\item\label{main-5-iv} $\limsup_{t \to 0} \frac{1}{f(t)} \sup_{s \leq t} |X_s-x|=0$ $\mbb{P}^x$-almost surely.
	\end{enumerate}
\end{theorem}

\begin{remark} \label{main-6} \begin{enumerate}[wide, labelwidth=!, labelindent=0pt]
	\item\label{main-6-i} The implications $\ref{main-5-ii} \implies \ref{main-5-iii}\implies\ref{main-5-iv}$ hold for \emph{any} L\'evy-type process, i.e.\ all the additional assumptions are only needed for the proof of the implication $\ref{main-5-i}\implies\ref{main-5-ii}$.
	\item\label{main-6-ii} The implication $\ref{main-5-iii} \implies \ref{main-5-iv}$ holds for any strong Markov process, see Theorem~\ref{up-4}.
	\item\label{main-6-iii} In Theorem~\ref{main-5} we assume that the symbol $q$ satisfies the sector condition, cf.\ \eqref{def-eq14}. A close look at the proof shows that we actually only need a local sector condition, in the sense that, for fixed $x \in \mbb{R}^d$ there are constants $r>0$ and $C>0$ such that
	\begin{equation*}
		|\im q(z,\xi)| \leq C \re q(z,\xi) \fa \xi \in \mbb{R}^d, |z-x| \leq r.
	\end{equation*}
	The same is true for the Proposition~\ref{main-8} and Theorem~\ref{main-9} below.
	\item\label{main-6-iv} As already mentioned, assumption \eqref{A1'} is crucial for the proof of the implication $\ref{main-5-i}\implies\ref{main-5-ii}$. One might ask whether this implication also holds under assumption \eqref{A2} from Theorem~\ref{main-3}. We did not manage to prove this, but there is the following slightly weaker result. Consider
	\begin{equation}
		\int_0^1 \sup_{|z-x| \leq R}  \nu(z,\{|y| \geq c f(t)\}) \, dt < \infty \quad \text{for some $c>0$, $R>0$.}
		\tag{LTP1'} \label{LTP1'}
	\end{equation}	
	Note that \eqref{LTP1'} is a bit stronger than \ref{main-5-i}. If \eqref{A2} holds and 
	\begin{equation}
		\forall r \in (0,R) \, \exists x_0 \in \overline{B(x,r)} \, \, \forall \xi \in \mbb{R}^d,|\xi| \geq 1\::\: \sup_{|z-x| \leq r} |q(z,\xi)| \leq |q(x_0,\xi)|,
		\label{main-eq15}
	\end{equation}
	then $\eqref{LTP1'} \implies \ref{main-5-ii}$. Because of the majorization condition \eqref{main-eq15}, the proof of this assertion is analogous to the case of L\'evy processes, see the proof of Theorem~\ref{main-3} and Lemma~\ref{p-1}.
\end{enumerate}
\end{remark}

For the particular case that $(X_t)_{t \geq 0}$ is a L\'evy process, Theorem~\ref{main-5} yields the implications $\ref{main-3-i}\implies\ref{main-3-ii}\implies\ref{main-3-iii}\implies\ref{main-3-v}$ from Theorem~\ref{main-3}. In this sense, Theorem~\ref{main-5} is a natural extension of Theorem~\ref{main-3}. Unlike in the L\'evy case, it is not to be expected that the conditions \ref{main-5-i}-\ref{main-5-iv} in Theorem~\ref{main-5} are equivalent for a general L\'evy-type process. However, there is the following partial converse.

\begin{proposition}\label{main-7}
	Let $(X_t)_{t \geq 0}$ be a L\'evy-type process with characteristics $(b(x),0,\nu(x,dy))$, and let $f:[0,1] \to [0,\infty)$ be non-decreasing. \begin{enumerate}
		\item\label{main-7-i}If $x \in \mbb{R}^d$ is such that
		\begin{equation*}
			\int_0^1 \sup_{|z-x| \leq f(t)} \mbb{P}^z \left( \sup_{s \leq t} |X_s-z| \geq f(t) \right) \frac{1}{t} \, dt < \infty,
		\end{equation*}
		then
		\begin{equation*}
			\int_0^1 \inf_{|z-x| \leq 10f(t)} \nu(z; \{|y| > 10 f(t)\}) \, dt < \infty.
		\end{equation*}
		\item\label{main-7-ii} Assume that \eqref{A1'} from Theorem~\ref{main-5} holds for some $R>0$ and $x \in \mbb{R}^d$. If 
		\begin{equation*}
			\int_0^1 \inf_{|z-x| \leq C f(t)} \nu(z; \{|y| > C f(t)\}) \, dt < \infty
		\end{equation*}
		for a constant $C>0$, then \begin{equation*}
			\int_0^1 \inf_{|z-x| \leq C f(t)} \sup_{|\xi| \leq c/f(t)} \re q(z,\xi) \, dt < \infty
		\end{equation*}
		for all $c>0$.
	\end{enumerate}
\end{proposition}

The next result gives a lower bound for the growth of the sample paths of a L\'evy-type process.

\begin{proposition} \label{main-8}
	Let $(X_t)_{t \geq 0}$ be a L\'evy-type process with symbol $q$ satisfying the sector condition. Let $x \in \mbb{R}^d$. If $f:[0,1] \to (0,\infty)$ is a function such that
	\begin{equation}
		\limsup_{t \to 0} t \cdot \inf_{|z-x| \leq Rf(t)} \sup_{|\xi| \leq 1/(Cf(t))} \re q(z,\xi) = \infty,
		\label{main-eq17}
	\end{equation}
	for every $R\geq 1$ and some constant $C=C(R)>0$, then
	\begin{equation}
		\limsup_{t \to 0} \frac{1}{f(t)} \sup_{s \leq t} |X_s-x| =\infty \quad \text{$\mbb{P}^x$-a.s.}
		\label{main-eq18}
	\end{equation} 
	Moreover: \begin{enumerate}
		\item\label{main-8-i} If additionally $f$ is non-decreasing, then
		\begin{equation*}
			\limsup_{t \to 0} \frac{1}{f(t)} |X_t-x| =\infty \quad \text{$\mbb{P}^x$-a.s.}
		\end{equation*} 
		\item\label{main-8-ii} If $f$ is regularly varying at zero, then \eqref{main-eq18} holds under the milder assumption that \eqref{main-eq17} is satisfied for $R=1$.
	\end{enumerate}
\end{proposition}

In \cite[Theorem 4.3]{rs-growth}, this result was shown for power functions $f(t)=t^{\kappa}$ but in fact the proof goes through for arbitrary functions $f$, see Section~\ref{conv}. The following statement is an immediate consequence of Proposition~\ref{main-8}: If $f$ is a non-negative function and $q$ the symbol of a L\'evy-type process $(X_t)_{t \geq 0}$ such that
\begin{equation*}
	\limsup_{t \to 0} t^{-\beta/\alpha} f(t)<\infty 
	\quad \text{and} 
	\quad \inf_{|z-x| \leq r} \re q(z,\xi) \geq c |\xi|^{\alpha} 
	\quad \text{for all} \, \, |\xi| \gg 1
\end{equation*}
for some $x \in \mbb{R}^d$, $r>0$, $\alpha>0$, $c>0$ and $\beta>1$, then
\begin{equation*}
	\limsup_{t \to 0} \frac{1}{f(t)} \sup_{s \leq t} |X_s-x| = \infty \quad \text{$\mbb{P}^x$-a.s.}
\end{equation*}

Our final main result gives an integral criterion for a function $f$ \emph{not} to be an upper function of a L\'evy-type process.

\begin{theorem} \label{main-9}
	Let $(X_t)_{t \geq 0}$ be a L\'evy-type process with characteristics $(b(x),0,\nu(x,dy))$ and symbol $q$, and let $f:[0,1] \to (0,\infty)$ be non-decreasing. Let $x \in \mbb{R}^d$ be such that one of the following conditions holds. 
	\begin{enumerate}[label*=\upshape (C\arabic*),ref=\text{\upshape (C\arabic*)}]
		\item\label{C1} $q$ satisfies the sector condition and there is $\kappa \in [0,1)$ such that 
		\begin{equation}
			\sup_{|z-x| \leq f(t)} \sup_{|\xi| \leq 1/f(t)} |q(z,\xi)|
			\leq c t^{-\kappa} \inf_{|z-x| \leq R f(t)} \sup_{|\xi| \leq 1/f(t)} |q(z,\xi)|, \qquad t \in (0,1),
			\label{main-eq185}
		\end{equation}
		for every $R \geq 1$ and some constant $c=c(R)>0$.
		\item\label{C2} There are constants $\alpha \in (0,2]$, $r>0$ and $c>0$ such that
		\begin{equation}
			\sup_{|z-x| \leq r} |q(z,\xi)| \leq c(1+|\xi|^{\alpha}), \qquad |\xi| \gg 1,
		\label{main-eq19}
		\end{equation}
		and $\liminf_{t \to 0} t^{-2/\alpha} f(t)=\infty$.
	\end{enumerate}
	Then: \begin{enumerate}
		\item\label{main-9-i} If 
		\begin{equation}
			\int_{0}^1 \inf_{|z-x| \leq C f(t)} \nu(z,\{|y| \geq C f(t)\}) \, dt = \infty \label{main-eq20}
		\end{equation}
		for some constant $C>0$, then
		\begin{equation*}
			\limsup_{t \to 0} \frac{1}{f(t)} |X_t-x| \geq \frac{C}{5} \quad \text{$\mbb{P}^x$-a.s.}
		\end{equation*}
		\item\label{main-9-ii} Assume that \eqref{A1'} from Theorem~\ref{main-5} is satisfied for some $R>0$. If 
		\begin{equation}
			\int_0^1 \inf_{|z-x| \leq C f(t)} \sup_{|\xi| \leq 1/f(t)} |q(z,\xi)| \, dt = \infty \label{main-eq21}
		\end{equation}
		for some constant $C>0$, then
		\begin{equation*}
			\limsup_{t \to 0} \frac{1}{f(t)} |X_t-x| \geq \frac{C}{5} \quad \text{$\mbb{P}^x$-a.s.}
		\end{equation*}
	\end{enumerate}
\end{theorem}

\begin{remark}\label{main-10}\begin{enumerate}[wide, labelwidth=!, labelindent=0pt]
	\item\label{main-10-iii} If the constant $C$ in \eqref{main-eq20} resp.\ \eqref{main-eq21} can be chosen arbitrarily large, then
		\begin{equation*}
			\limsup_{t \to 0} \frac{1}{f(t)} |X_t-x|=\infty \quad \text{$\mbb{P}^x$-a.s.}
		\end{equation*}	
	\item\label{main-10-i} The growth condition \eqref{main-eq19} on the symbol holds automatically for $\alpha=2$, cf.\ \eqref{def-eq13}.
	\item\label{main-10-ii} If $f$ is regularly varying, then it suffices to check \eqref{main-eq185} in \ref{C1} for $R=1$. Moreover, we note that \eqref{main-eq185} is trivially satisfied if the symbol $q$ does not depend on $z$, i.e.\ if $(X_t)_{t \geq 0}$ is a L\'evy process.
	\item\label{main-10-iv} For the particular case of L\'evy processes, Theorem~\ref{main-9} is known -- see Theorem~\ref{main-3} and the references below it -- but Theorem~\ref{main-9} seems to be the first result in this direction which applies for the much wider class of L\'evy-type processes. Let us comment on the differences in the proofs. For L\'evy processes, the standard approach to prove an assertion of the form
	\begin{equation*}
		\limsup_{t \to 0} \frac{1}{f(t)} |X_t| \geq C \quad \text{a.s.}
	\end{equation*}
	is to construct a suitable sequence $(A_n)_{n \in \mbb{N}}$ of sets which involves only the increments of $(X_t)_{t \geq 0}$ and which satisfies
	\begin{equation*}
		\limsup_{n \to \infty} A_n \subseteq \left\{ \limsup_{t \to 0} \frac{1}{f(t)} |X_t| \geq C \right\},
	\end{equation*}
	and to use (the difficult direction of) the Borel--Cantelli lemma to deduce from $\sum_{n \in \mbb{N}} \mbb{P}(A_n)=\infty$ that $\mbb{P}(\limsup_n A_n)=1$. This approach relies heavily on the independence of the increments -- ensuring the sets $(A_n)_{n \in \mbb{N}}$ are independent -- and so it fails to work in the more general framework of L\'evy-type processes. We fix this issue by using a conditional Borel-Cantelli lemma for backward filtrations, cf.\ Proposition~\ref{appendix-1}. Moreover, our proof uses a new maximal inequality for L\'evy-type processes which is of independent interest, cf.\ Section~\ref{max}.
\end{enumerate} \end{remark}

We close this section with some illustrating examples.

\begin{example}[Process of variable order] \label{main-11} 
	Let $(X_t)_{t \geq 0}$ be a L\'evy-type process with symbol $q(x,\xi)=|\xi|^{\alpha(x)}$ for $\alpha: \mbb{R}^d \to (0,2)$ continuous; a sufficient condition for the existence of such a process is that $\alpha$ is H\"older continuous and bounded away from zero, see e.g.\ \cite{bass,kol,matters} for details. Let us mention that Negoro \cite{negoro94} was one of the first to study the small-time asymptotics of processes of variable order. If we set $\alpha^*(x,r):=\sup_{|z-x| \leq r} \alpha(z)$ and $\alpha_*(x,r) := \inf_{|z-x| \leq r} \alpha(z)$, then our results show that
	\begin{equation}
		\int_0^1 |f(t)|^{-\alpha^*(x,r)} \, dt < \infty \, \, \text{for some $r>0$} \implies \limsup_{t \to 0} \frac{1}{f(t)} \sup_{s \leq t} |X_s-x| =0 \quad \text{$\mbb{P}^x$-a.s.}
		\label{main-eq25}
	\end{equation}
	and 
	\begin{equation}
		\int_0^1 |f(t)|^{-\alpha_*(x,r)} \, dt = \infty \, \, \text{for some $r>0$} \implies \limsup_{t \to 0} \frac{1}{f(t)} |X_t-x| =\infty \quad \text{$\mbb{P}^x$-a.s.} \label{main-eq26}
	\end{equation}
	for any $f:[0,1] \to [0,\infty)$ non-decreasing. By the continuity of $\alpha$, this entails that
	\begin{equation*}
		\int_0^1 |f(t)|^{-\beta} \, dt < \infty \, \, \text{for some $\beta>\alpha(x)$} \implies \limsup_{t \to 0} \frac{1}{f(t)} \sup_{s \leq t} |X_s-x| =0 \quad \text{$\mbb{P}^x$-a.s.}
	\end{equation*}
	and
	\begin{equation*}
		\int_0^1 |f(t)|^{-\beta} \, dt = \infty \, \, \text{for some $\beta<\alpha(x)$} \implies \limsup_{t \to 0} \frac{1}{f(t)}  |X_t-x| =\infty \quad \text{$\mbb{P}^x$-a.s.}
	\end{equation*}
	In particular, this generalizes \cite[Theorem 2.1]{negoro94}, which is about the particular case that $f$ is a power function. If $\alpha$ has a local maximum at $x$, then \eqref{main-eq25} yields
		\begin{equation*}
			\int_0^1 |f(t)|^{-\alpha(x)} \, dt < \infty \implies \limsup_{t \to 0} \frac{1}{f(t)} \sup_{s \leq t} |X_s-x| =0 \quad \text{$\mbb{P}^x$-a.s.}
		\end{equation*}
	This holds, in particular, if $\alpha(x)=\alpha$ is constant, i.e.\ if $(X_t)_{t \geq 0}$ is an isotropic $\alpha$-stable L\'evy process. An analogous consideration works for \eqref{main-eq26} if $\alpha$ has a local minimum at $x$. In particular, we recover the classical criterion for isotropic $\alpha$-stable L\'evy processes, cf.\ Example~\ref{main-45}.
\end{example}

\begin{example}[Stable-type process] \label{main-12}
	Consider a L\'evy-type process $(X_t)_{t \geq 0}$ with characteristics $(0,0,\nu(x,dy))$, where
	\begin{equation*}
		\nu(x,dy) = \kappa(x,y) \frac{1}{|y|^{d+\alpha}} \,dy
	\end{equation*}
	for some $\alpha \in (0,2)$ and a mapping $\kappa: \mbb{R}^d \times \mbb{R} \to (0,\infty)$ which is symmetric in the $y$-variable and satisfies $0<\inf_{x,y} \kappa(x,y)\leq\sup_{x,y} \kappa(x,y) < \infty$, see e.g.\ \cite{bass09,matters} for the existence of such processes. Since
	\begin{equation*}
		\frac{1}{M} r^{-\alpha} \leq \nu(x,\{|y| \geq r\}) \leq M r^{-\alpha}, \qquad r>0,
	\end{equation*}
	for a constant $M>0$ not depending on $x \in \mbb{R}^d$, it follows from Theorem~\ref{main-5} and Theorem~\ref{main-9} that
	\begin{equation*}
		\limsup_{t \to 0} \frac{1}{f(t)} |X_t-x| = \begin{rcases}
	    \begin{dcases}0 \\ \infty \end{dcases} \end{rcases} \, \, \text{$\mbb{P}^x$-a.s.} 
		\quad \text{according as} \quad
		\int_0^1 |f(t)|^{-\alpha} \, dt \begin{rcases}
	    \begin{dcases}< \infty \\ = \infty \end{dcases} \end{rcases}
	\end{equation*}
	for any non-decreasing function $f:[0,1] \to [0,\infty)$.
\end{example}

\begin{example}[{L\'evy-driven SDE}] \label{main-13}
	Let $(L_t)_{t \geq 0}$ be a pure-jump L\'evy process with characteristic exponent $\psi$ satisfying the sector condition, and assume that the L\'evy measure $\nu_L$ satisfies \eqref{A1} from Theorem~\ref{main-3}. Let $(X_t)_{t \geq 0}$ be the unique weak solution to an SDE
	\begin{equation*}
		dX_t = \sigma(X_{t-}) \, dL_t, \qquad X_0 = x,
	\end{equation*}
	for a bounded continuous function $\sigma: \mbb{R} \to \mbb{R}$. Then $(X_t)_{t \geq 0}$ is a L\'evy-type process with symbol $q(x,\xi) := \psi(\sigma(x) \xi)$, cf.\ \cite{kurtz,sde,schnurr10}. Fix $x \in \mbb{R}$ such that $\sigma(x) \neq 0$. By Theorem~\ref{main-5} and Theorem~\ref{main-9}, the following statements hold for any non-decreasing function  $f:[0,1] \to [0,\infty)$. \begin{enumerate}
		\item\label{main-13-i}  If there exists a constant $c>0$ such that \begin{equation*}
			\int_0^1 \nu_L(\{|y| \geq cf(t)\}) \, dt < \infty,
		\end{equation*}
		then
		\begin{equation*}
			\limsup_{t \to 0} \frac{1}{f(t)} \sup_{s \leq t} |X_s-x| = 0 \quad \text{and} \quad
			\limsup_{t \to 0} \frac{1}{f(t)} \sup_{s \leq t} |L_s| = 0.
		\end{equation*}
		\item\label{main-13-ii} Assume that $\psi^*(r):=\sup_{|\xi| \leq r} |\psi(\xi)|$ satisfies the following weak scaling condition (at zero): There are constants $\alpha>0$ and $C>0$ such that
		\begin{equation*}
			\psi^*(\lambda r) \geq C \lambda^{\alpha} \psi^*(r) \fa r>0,\;\lambda \in (0,1).
		\end{equation*}
		If
		\begin{equation*}
			\int_0^1 \nu_L(\{|y| \geq c f(t)\}) \, dt = \infty
		\end{equation*}
		for some constant $c>0$, then
		\begin{equation*}
			\limsup_{t \to 0} \frac{1}{f(t)} \sup_{s \leq t} |X_s-x| > 0 \quad \text{and} \quad
			\limsup_{t \to 0} \frac{1}{f(t)} \sup_{s \leq t} |L_s| > 0.
		\end{equation*}
	\end{enumerate}
\end{example}

The remainder of the article is organized as follows. After introducing basic definitions and notation in Section~\ref{def}, we establish a new maximal inequality for L\'evy-type processes in Section~\ref{max} and study some of its consequences. In Section~\ref{up} we obtain integral criteria for upper functions of Markov processes. They are the key for the proofs of Theorem~\ref{main-3} and Theorem~\ref{main-5}, which are presented in Section~\ref{p}. Finally, in Section~\ref{conv}, we give the proofs of Proposition~\ref{main-7}, Proposition~\ref{main-8} and Theorem~\ref{main-9}.

\section{Basic definitions and notation} \label{def}

We consider the Euclidean space $\mbb{R}^d$ with the canonical scalar product $x \cdot y := \sum_{j=1}^d x_j y_j$ and the Borel $\sigma$-algebra $\mc{B}(\mbb{R}^d)$ generated by the open balls $B(x,r) := \{y \in \mbb{R}^d; |y-x|<r\}$. For a real-valued function $f$, we denote by $\nabla f$ the gradient and by $\nabla^2 f$ the Hessian of $f$. If $\nu$ is a measure, say on $\mbb{R}^d$, we use the short-hand $\nu(\{|y|>r\})$ for $\nu(\{y \in \mbb{R}^d; |y|>r\})$. \par

An operator $A$ defined on the space $C_c^{\infty}(\mbb{R}^d)$ of compactly supported smooth functions is a \emph{L\'evy-type operator} if it has a representation of the form
\begin{align} \label{def-eq3} \begin{aligned}
	Af(x) 
	&= b(x) \cdot \nabla f(x) + \frac{1}{2} \tr(Q(x) \cdot \nabla^2 f(x)) \\
	&\quad + \int_{y \neq 0} (f(x+y)-f(x)-y \cdot \nabla f(x) \I_{(0,1)}(|y|)) \, \nu(x,dy), \qquad  f\in C_c^{\infty}(\mbb{R}^d), 
\end{aligned}
\end{align}
where $b(x) \in \mbb{R}^d$ is a vector, $Q(x) \in \mbb{R}^{d \times d}$ is a positive semi-definite matrix and $\nu(x,dy)$ is a measure on $\mbb{R}^d \setminus \{0\}$ satisfying $\int_{y \neq 0} \min\{1,|y|^2\} \, \nu(x,dy) < \infty$ for each $x \in \mbb{R}^d$. The family $(b(x),Q(x),\nu(x,dy))$, $x \in \mbb{R}^d$, is the \emph{(infinitesimal) characteristics} of $A$. Equivalently, $A$ can be written as a pseudo-differential operator
\begin{equation*}
	Af(x) = - \int_{\mbb{R}^d} q(x,\xi) e^{ix \cdot \xi} \hat{f}(\xi) \, d\xi, \qquad f \in C_c^{\infty}(\mbb{R}^d),\; x \in \mbb{R}^d,
\end{equation*}
with \emph{symbol} 
\begin{equation*}
	q(x,\xi) := -i b(x) \cdot \xi + \frac{1}{2} \xi \cdot Q(x) \xi + \int_{y \neq 0} \left(1-e^{iy \cdot \xi}+iy \cdot \xi \I_{(0,1)}(|y|) \right) \, \nu(x,dy), \qquad x,\xi \in \mbb{R}^d.
\end{equation*}
For each fixed $x \in \mbb{R}^d$, the mapping $\xi \mapsto q(x,\xi)$ is continuous and negative definite (in the sense of Schoenberg). In consequence, $\xi \mapsto \sqrt{|q(x,\xi)|}$ is subadditive, i.e.
\begin{equation}
	\sqrt{|q(x,\xi+\eta)|} \leq \sqrt{|q(x,\xi)|} + \sqrt{|q(x,\eta)|}, \qquad x,\xi,\eta \in \mbb{R}^d,
	\label{def-eq11}
\end{equation}
which implies
\begin{equation*}
	|q(x,\xi)| \leq 2 \sup_{|\eta| \leq 1} |q(x,\eta)|  (1+|\xi|^2), \qquad x,\xi \in \mbb{R}^d,
\end{equation*}
see e.g.\ \cite[Theorem 6.2]{barca}. In particular, $(x,\xi) \mapsto q(x,\xi)$ is locally bounded if, and only if, there is for every compact set $K \subseteq \mbb{R}^d$ some constant $C>0$ such that \begin{equation}
	|q(x,\xi)| \leq C(1+|\xi|^2), \qquad \xi \in \mbb{R}^d,\; x \in K. \label{def-eq13}
\end{equation}
The local boundedness of $q$ can also be characterized in terms of the characteristics; namely, $q$ is locally bounded if, and only if,
\begin{equation*}
	\forall K \subseteq \mbb{R}^d \, \, \text{compact}\::\: \sup_{x \in K} \left( |b(x)|+|Q(x)| +\int_{y \neq 0} \min\{1,|y|^2\} \nu(x,dy) \right)<\infty.
\end{equation*}
Applying Taylor's formula in \eqref{def-eq3} shows that the local boundedness of the symbol $q$ of the L\'evy-type operator $A$ implies $\|Af\|_{\infty}<\infty$ for every $f \in C_c^{\infty}(\mbb{R}^d)$. We say that $q$ satisfies the \emph{sector condition} if there is a constant $C>0$ such that \begin{equation}
	|\im q(x,\xi)| \leq C \re q(x,\xi) \fa x,\xi \in \mbb{R}^d. \label{def-eq14}
\end{equation}
Next we introduce the probabilistic objects. Let $A$ be a L\'evy-type operator and $(\Omega,\mc{A},\mbb{P})$ a probability space. A stochastic process $X_t: \Omega \to \mbb{R}^d$, $t \geq 0$, with \cadlag sample paths is a \emph{solution to the $(A,C_c^{\infty}(\mbb{R}^d))$-martingale problem with initial distribution $\mu$} if $\mbb{P}(X_0 \in \cdot)=\mu(\cdot)$ and
\begin{equation*}
	M_t^f := f(X_t)-f(X_0)- \int_0^t Af(X_s) \, ds, \qquad t \geq 0,
\end{equation*}
is a martingale with respect to the canonical filtration $\mc{F}_t := \sigma(X_s; s \leq t)$ for every $f \in C_c^{\infty}(\mbb{R}^d)$. A tuple $(X_t, t \geq 0; \mbb{P}^x, x \in \mbb{R}^d)$ consisting of a family of probability measures $\mbb{P}^x$, $x \in \mbb{R}^d$, on a measurable space $(\Omega,\mc{A})$ and a stochastic process $X_t: \Omega \to \mbb{R}^d$ with \cadlag sample paths is called a  \emph{L\'evy-type process with symbol $q$} if
\begin{enumerate}
	\item\label{ltp-1} $(X_t,t \geq 0; \mbb{P}^x, x \in \mbb{R}^d)$ is a strong Markov process;
	\item\label{ltp-2} $(X_t)_{t \geq 0}$ solves the $(A,C_c^{\infty}(\mbb{R}^d))$-martingale problem for the L\'evy-type operator $A$ with symbol $q$. More precisely, for each $x \in \mbb{R}^d$, the stochastic process $(X_t)_{t \geq 0}$ considered on the probability space $(\Omega,\mc{A},\mbb{P}^x)$ is a solution to the $(A,C_c^{\infty}(\mbb{R}^d))$-martingale problem with initial distribution $\mu=\delta_x$;
	\item\label{ltp-3} $q$ is locally bounded.
\end{enumerate}
Note that \ref{ltp-3} entails that $Af$ is bounded for every $f \in C_c^{\infty}(\mbb{R}^d)$, and so the integral $\int_0^t Af(X_s) \, ds$ appearing in the definition of the martingale problem is well-defined. If the $(A,C_c^{\infty}(\mbb{R}^d))$-martingale problem is \emph{well-posed}, i.e.\ there exists a unique solution to the martingale problem for any initial distribution $\mu$, then the strong Markov property \ref{ltp-1} is automatically satisfied, cf.\ \cite[Theorem 4.4.2]{ethier}. Well-posedness is, however, not necessary for the existence of strongly Markovian solutions to martingale problems; one can use so-called Markovian selections to construct such solutions, see \cite[Section 4.5]{ethier} and \cite{markovian}. For a thorough discussion of martingale problems associated with L\'evy-type operators, we refer to \cite{ltp,hoh,jacob1}. The following classes of stochastic processes are examples of L\'evy-type processes: \begin{itemize}
	\item L\'evy processes: A \emph{L\'evy process} is a stochastic process $(X_t)_{t \geq 0}$ with stationary and independent increments and \cadlag sample paths. It is uniquely determined (in distribution) by its L\'evy triplet $(b,Q,\nu)$ and its characteristic exponent $\psi$, cf.\ \cite{sato,barca}. Any L\'evy process is a L\'evy-type process in the sense of the above definition; the corresponding operator $A$ is the pseudo-differential operator with symbol $q(x,\xi):=\psi(\xi)$ and characteristics $(b,Q,\nu)$.
	\item Feller processes: If $(X_t)_{t \geq 0}$ is a Feller process whose infinitesimal generator $(A,\mc{D}(A))$ satisfies $C_c^{\infty}(\mbb{R}^d) \subseteq \mc{D}(A)$, then $(X_t)_{t \geq 0}$ is a L\'evy-type operator; this follows from a result by Courr\`ege \& von Waldenfels, see \cite{ltp,jacob1,mp-feller} for further information.
	\item solutions to L\'evy-driven SDEs: Let $(L_t)_{t \geq 0}$ be a L\'evy process with characteristic exponent $\psi$. If $\sigma$ is a bounded continuous function and the stochastic differential equation (SDE)
	\begin{equation*}
		dX_t = \sigma(X_{t-}) dL_t, \qquad X_0 =x,
	\end{equation*}
	has a unique weak solution $(X_t)_{t \geq 0}$, then $(X_t)_{t \geq 0}$ is a L\'evy-type process with symbol $q(x,\xi)=\psi(\sigma(x)^T \xi)$, cf.\ \cite{ethier,schnurr10}. The assumptions on $\sigma$ can be relaxed, cf.\ \cite{sde,markovian}.
\end{itemize}

\section{A maximal inequality for L\'evy-type processes} \label{max}

In this section, we establish a new maximal inequality for L\'evy-type processes and present some consequences of this inequality. Before we start, we recall the following maximal inequality, which will be used frequently in this paper.

\begin{proposition} \label{max-0}
	Let $(X_t)_{t \geq 0}$ be a L\'evy-type process with symbol $q$. There is an absolute constant $c>0$ such that
	\begin{equation}
		\mbb{P}^x \left( \sup_{s \leq t} |X_s-x| \geq r \right) \leq ct \sup_{|z-x| \leq r} \sup_{|\xi| \leq 1/r} |q(z,\xi)|, \qquad x \in \mbb{R}^d,\;t>0,\;r>0. \label{max-eq0}
	\end{equation}
\end{proposition}

This maximal inequality goes back to Schilling \cite{rs-growth}, see also \cite[Theorem 5.1]{ltp} and \cite[Proposition 2.8]{markovian}. Let us mention two variants of this inequality: a version for random times, cf.\ \cite[Lemma 1.29]{matters}, and a localized version, cf.\ \cite[Lemma 4.1]{ihke}. If we denote by $\tau_r^x = \inf\{t \geq0; |X_t-x| \geq r\}$ the first exit time of $(X_t)_{t \geq 0}$ from the open ball $B(x,r)$, then \eqref{max-eq0} can be equivalently formulated as follows:
\begin{equation*}
	\mbb{P}^x(\tau_r^x \leq t) 
	\leq c t \sup_{|z-x| \leq r} \sup_{|\xi| \leq 1/r} |q(z,\xi)|, \qquad x \in \mbb{R}^d,\;t>0,\;r>0.
\end{equation*}

For the proofs of our main results, we need an upper bound for the probability
\begin{equation*}
	\mbb{P}^x \left( \sup_{s \leq t} |X_s-x| < r \right).
\end{equation*}
The following maximal inequality allows us to derive suitable bounds and is of independent interest.

\begin{proposition} \label{max-1}
	Let $(X_t)_{t \geq 0}$ be a L\'evy-type process with characteristics $(b(x),Q(x),\nu(x,dy))$, $x \in \mbb{R}^d$, and denote by $\tau_r^x$ the first exit time from $B(x,r)$. Then
	\begin{equation*}
		\mbb{P}^x(\tau_r^x \geq t) \leq \frac{1}{1+t G(x,2r)} \quad \text{with} \quad G(x,r) := \inf_{|z-x| \leq r} \nu(z, \{|y| > r\})
	\end{equation*}
	for all $x \in \mbb{R}^d$ and $r>0$.
\end{proposition}

Note that Proposition~\ref{max-1} implies that
\begin{equation*}
	\mbb{P}^x \left( \sup_{s \leq t} |X_s-x| < r \right) \leq \frac{1}{1+tG(x,2r)}, \qquad x \in \mbb{R}^d,\; r>0.
\end{equation*}
Intuitively, $G(x,r) = \inf_{|z-x| \leq r} \nu(z,\{|y| > r\})$ quantifies the likelihood of a jump of modulus $> r$ while the process is close to its starting point $x$. The idea behind our estimate is that the process leaves immediately the ball $B(x,r)$ if a jump of modulus $>2r$ occurs. Other means of leaving the ball, e.g.\ due to a drift or diffusion part, are not taken into account. In consequence, Proposition~\ref{max-1} does well for pure-jump processes but less so e.g.\ for processes with a non-vanishing diffusion part. There is a related estimate in \cite[Theorem 5.5]{ltp}, see also \cite[Lemma 6.3]{rs-growth}, giving an upper bound for $\mbb{P}^x(\tau_r^x \geq t)$ in terms of the symbol\footnote{Beware that there is a typo in the definition of $k(x,r)$ in \cite[Theorem 5.5]{ltp}; the two suprema should be infima.}; for the particular that $q$ satisfies the sector condition \eqref{def-eq14}, it reads
\begin{equation}
	\mbb{P}^x(\tau_r^x \geq t) \leq c \frac{1}{1+t h(x,r)} \quad \text{with} \quad h(x,r) := \sup_{|\xi| \leq 1/(2r)} \inf_{|z-x| \leq r} \re q(z,\xi) \label{max-eq4}
\end{equation}
for some constant $c>0$. In some situations, \eqref{max-eq4} gives better estimates than Proposition~\ref{max-1} -- e.g.\ if there is a diffusion part -- but our result has its advantages e.g.\ if the sector condition is not satisfied. For instance, for $q(x,\xi) = i \xi + \sqrt{|\xi|}$, we get $\mbb{P}^x(\tau_r^x \geq t) \leq 1/(1+ct r^{-1/2}) \leq c' r^{1/2} t^{-1}$ while \cite[Theorem 5.5]{ltp} gives only $\mbb{P}^x(\tau_r^x \geq t) \leq c'' r^{1/3}t^{-1}$; note that the estimates are of interest only if the right-hand sides are less or equal than $1$, i.e.\ for $r>0$ small.

\begin{proof}[Proof of Proposition~\ref{max-1}]
	For fixed $x \in \mbb{R}^d$, $\eps>0$ and $r>0$, pick $\chi \in C_c^{\infty}(\mbb{R}^d)$ such that $\I_{B(x,r)} \leq \chi \leq \I_{B(x,r+\eps)}$. As $\chi(X_t)=1$ on $\{t<\tau_r^x\}$, it follows from Dynkin's formula that
	\begin{equation}
		\mbb{P}^x(\tau_r^x>t) 
		= \mbb{E}^x(\chi(X_t) \I_{\{\tau_r^x>t\}})
		\leq \mbb{E}^x(\chi(X_{t \wedge \tau_r^x}))
		= 1 + \mbb{E}^x \left( \int_{(0,t \wedge \tau_r^x)} A \chi(X_s) \, ds \right),
		\label{max-eq5}
	\end{equation}
	where $A$ is the L\'evy-type operator associated with the family of triplets $(b(x),Q(x),\nu(x,dy))$, see \eqref{def-eq3}. For $z \in B(x,r)$, we have $\chi(z)=1$, $\nabla \chi(z)=0$ and $\nabla^2 \chi(z)=0$, and so
	\begin{equation*}
		A\chi(z) = \int_{y \neq 0} (\chi(z+y)-1) \, \nu(z,dy).
	\end{equation*}
	Using that $0 \leq \chi \leq 1$ on $\mbb{R}^d$ and $\chi=0$ outside $B(x,r+\eps)$, we find that
	\begin{equation*}
		A \chi(z) 
		\leq \int_{|(z+y)-x| \geq r+\eps} (\chi(z+y)-1) \, \nu(z,dy)
		\leq - \int_{|y| \geq 2r+\eps} \nu(z,dy)
	\end{equation*}
	for all $z \in B(x,r)$. Since $X_s \in B(x,r)$ for $s<\tau_r^x$, it follows from \eqref{max-eq5} that
	\begin{equation*}
		\mbb{P}^x(\tau_r^x>t)
		\leq 1- \mbb{E}^x \left( \int_{(0,t \wedge \tau_r^x)} \nu(X_s,\{|y| \geq 2r+\eps\}) \, ds \right)
	\end{equation*}
	for all $\eps>0$. Letting $\eps \downarrow 0$ using the dominated convergence theorem, we arrive at	
	\begin{align}
		\mbb{P}^x(\tau_r^x>t) 
		&\leq 1-  \mbb{E}^x \left( \int_{(0,t \wedge \tau_r^x)} \nu(X_s,\{|y| >2r \}) \, ds \right) \notag\\
		&\leq 1- \mbb{E}^x(\tau_r^x \wedge t) \inf_{|z-x| \leq r} \int_{|y| > 2r} \nu(z,dy).
		\label{max-eq7}
	\end{align}
	The elementary estimate $\mbb{E}^x(\tau_r^x \wedge t) \geq t \mbb{P}^x(\tau_r^x>t)$ now gives
	\begin{equation*}
		\mbb{P}^x(\tau_r^x>t) \leq 1- t \mbb{P}^x(\tau_r^x>t) \inf_{|z-x| \leq r} \int_{|y|>2r} \nu(z,dy),
	\end{equation*}
	i.e.
	\begin{align*}
		\mbb{P}^x(\tau_r^x>t) \leq \frac{1}{1+t\inf_{|z-x| \leq r} \nu(z,\{|y|>2r\})}.
	\end{align*}
	Thus,
	\begin{equation*}
		\mbb{P}^x(\tau_r^x \geq t)
		= \lim_{\eps \to 0} \mbb{P}^x(\tau_r^x > t-\eps)
		\leq  \frac{1}{1+t\inf_{|z-x| \leq r} \nu(z,\{|y|>2r\})}. \qedhere
	\end{equation*}
\end{proof}

\begin{corollary} \label{max-3}
	Let $(X_t)_{t \geq 0}$ be a L\'evy-type process with characteristics $(b(x),Q(x),\nu(x,dy))$, and denote by $\tau_r^x$ the exit time from the ball $B(x,r)$. Then
	\begin{equation}
		\mbb{E}^x \tau_r^x \leq \frac{1}{G(x,2r)} \label{max-eq11}
	\end{equation}
	and
	\begin{equation}
		\mbb{P}^x(\tau_r^x \geq t) \leq C_0 \exp (-C_1 t G(x,2r)) \label{max-eq12}
	\end{equation}
	for all $x \in \mbb{R}^d$, $t>0$ and $r>0$, where $C_0,C_1 <\infty$ are uniform constants and $G(x,r)$ is the mapping defined in Proposition~\ref{max-1}.
\end{corollary}

\begin{proof}
	By \eqref{max-eq7},
	\begin{equation*}
		\mbb{E}^x(\tau_r^x \wedge t) 
		\leq \frac{1-\mbb{P}^x(\tau_r^x>t)}{G(x,2r)}
		\leq \frac{1}{G(x,2r)}.
	\end{equation*}
	Letting $t \to \infty$ using Fatou's lemma, proves the first assertion. The second inequality is obtained from Proposition~\ref{max-1} by an iteration argument using the Markov property; it is the same reasoning as in \cite[Proof of Theorem 5.9]{ltp}.
\end{proof}

\begin{remark} \begin{enumerate}[wide, labelwidth=!, labelindent=0pt]
	\item For the particular case that $(X_t)_{t \geq 0}$ is a L\'evy process, we know that $N_t := \sharp\{s \leq t; |\Delta X_s| > 2r\}$ is a Poisson process with intensity $\lambda=\nu(\{|y|>2r\})$, where $\nu$ is the L\'evy measure, and so
	\begin{equation*}
		\mbb{P}^x(\tau_r^x \geq t) 
		= \lim_{\eps \to 0} \mbb{P}^x(\tau_r^x > t+\eps)
		\leq \lim_{\eps \to 0} \mbb{P}^x(N_{t+\eps}=0)
		= e^{-\lambda t} = e^{-t\nu(\{|y|>2r\})},
	\end{equation*}
	which is \eqref{max-eq12} with $C_0=C_1=1$. If $(X_t)_{t \geq 0}$ is a general L\'evy-type process $(X_t)_{t \geq 0}$, then $(N_t)_{t \geq 0}$ is no longer a Poisson process but our result shows that we can still get an analogous estimate in terms of the jump intensity $G(x,2r)$. This fits well to the intuition that a L\'evy-type process behaves locally like a L\'evy process.
	\item The estimate \eqref{max-eq12} is optimal for a wide family of jump processes. However, our approach incorporates only the tails of the L\'evy measures and therefore some information may be lost, leading to non-optimal estimates for certain processes. This is best seen for the particular case of stable L\'evy processes, for which Taylor \cite{taylor67} derived upper and lower bounds for $\mbb{P}(\tau_r \geq t)$ (i.e. $x=0$). He shows for $r>0$ small that
	\begin{align*}
		\mbb{P}(\tau_r\geq t) &\asymp e^{-ct r^{-\alpha}} \qquad \text{for stable processes of type A} \\ 
		\mbb{P}(\tau_r \geq t) &\asymp e^{- ct r^{-\alpha/(1-\alpha)}} \qquad \text{for stable processes of type B, $\alpha \in (0,1)$},
	\end{align*}
	where the constants $c$ in the lower and upper bound may differ. Here, 'type B' means essentially that the process has a projection which is a subordinator -- formally, the L\'evy measure is concentrated on a hemisphere $\{y \in \mbb{R}^d; y_j \geq 0\}$ for some $j \in \{1,\ldots,d\}$ -- and all other stable processes are of type A. While our estimate \eqref{max-eq12} yields the correct upper bound for stable processes of type A, we only get the (sub-optimal) upper bound $e^{-ctr^{-\alpha}}$ for processes of type B.
\end{enumerate}
\end{remark}

As a direct consequence of Proposition~\ref{max-1}, we also obtain the following corollary.

\begin{corollary} \label{max-5}
	Let $(X_t)_{t \geq 0}$ be a L\'evy-type process with characteristics $(b(x),Q(x),\nu(x,dy))$, and let $c \in [0,1]$. If $x \in \mbb{R}^d$, $t>0$ and $r>0$ are such that
	\begin{equation*}
		\mbb{P}^x \left( \sup_{s \leq t} |X_s-x| > r \right) \leq c,
	\end{equation*}
	then
	\begin{equation*}
		\mbb{P}^x \left( \sup_{s \leq t} |X_s-x| > r \right) \geq (1-c) t G(x,2r)
	\end{equation*}
	for $G(x,r)$ defined in Proposition~\ref{max-1}. 
\end{corollary}

As an immediate consequence, we see that
\begin{equation*}
	\limsup_{t \to 0} \mbb{P}^x \left( \sup_{s \leq t} |X_s-x| > r(t) \right) < 1
\end{equation*}
implies
\begin{equation*}
	\mbb{P}^x \left( \sup_{s \leq t} |X_s-x|>r(t) \right) \geq C t G(x,2r(t))
\end{equation*}
for small $t>0$ and some constant $C>0$, which will be useful lateron.

\begin{proof}[Proof of Corollary~\ref{max-5}]
	By Proposition~\ref{max-1}, 
	\begin{equation*}
		\mbb{P}^x \left( \sup_{s \leq t} |X_s-x| \leq r\right) \leq \frac{1}{1+tG(x,2r)},
	\end{equation*}
	which is equivalent to 
	\begin{equation*}
		\mbb{P}^x \left( \sup_{s \leq t} |X_s-x| \leq r\right) 
		\leq 1- t G(x,2r) \mbb{P}^x \left( \sup_{s \leq t} |X_s-x| \leq r\right).
	\end{equation*}
	Hence,
	\begin{align*}
		\mbb{P}^x \left( \sup_{s \leq t} |X_s-x|>r \right)
		&\geq t G(x,2r) \mbb{P}^x \left( \sup_{s \leq t} |X_s-x| \leq r\right) \\
		&= t G(x,2r) \left[ 1- 	\mbb{P}^x \left( \sup_{s \leq t} |X_s-x|>r \right) \right],
	\end{align*}
	which proves the assertion.
\end{proof}

Let us illustrate the results from this section with an example.

\begin{example} \label{max-7}
	Let $(X_t)_{t \geq 0}$ be a process of variable order, i.e.\ a L\'evy-type process with symbol $q(x,\xi) = |\xi|^{\alpha(x)}$ for a continuous mapping $\alpha: \mbb{R}^d \to (0,2]$. Denote by $\tau_r^x$ the first exit time of $(X_t)_{t \geq 0}$ from the ball $B(x,r)$ and set $\alpha_*(x,r):=\inf_{|z-x| \leq r} \alpha(z)$. The following estimates hold for uniform constants $c_0,\ldots,c_4 \in (0,\infty)$: \begin{enumerate}
		\item $\mbb{P}^x(\tau_r^x \geq t) \leq 1/(1+ c_0 t r^{-\alpha_*(x,r)})$ and $\mbb{P}^x(\tau_r^x \geq t) \leq c_1 \exp(-c_2 t r^{-\alpha_*(x,r)})$,
		\item $\mbb{E}^x(\tau_r^x) \leq c_3 r^{\alpha_*(x,r)}$,
		\item $\mbb{P}^x \left( \sup_{s \leq t} |X_s-x| \geq r\right) \geq c_4 t r^{-\alpha_*(x,r)}$ for $t=t(r)>0$ small.
	\end{enumerate}
\end{example}

\section{Integral criteria for upper functions} \label{up}

Let $(X_t)_{t \geq 0}$ be a Markov process and $f:[0,1] \to [0,\infty)$ a non-decreasing function. The aim of this section is to derive sufficient conditions for
\begin{equation}
	\limsup_{t \to 0} \frac{1}{f(t)} \sup_{s \leq t} |X_s-x| \leq c \quad \text{$\mbb{P}^x$-a.s.} 
	\label{up-eq1}
\end{equation}
in terms of certain integrals. Our first main result is the following theorem.

\begin{theorem} \label{up-4}
	Let $(X_t)_{t \geq 0}$ be a Markov process with \cadlag sample paths and $f:[0,1] \to [0,\infty)$ a non-decreasing function. If
	\begin{equation}
		\int_0^1 \frac{1}{t} \sup_{|z-x| \leq  f(t)} \mbb{P}^z \left( \sup_{s \leq t} |X_s-z| \geq f(t) \right) \, dt < \infty \label{up-eq3}
	\end{equation}
	for some $x \in \mbb{R}^d$, then
	\begin{equation*}
		\limsup_{t \to 0} \frac{1}{f(t)} \sup_{s \leq t} |X_s-x| \leq 4 \quad \text{$\mbb{P}^x$-a.s.}
	\end{equation*}
\end{theorem}

\begin{proof}
	\primolist Claim: \begin{equation}
		\mbb{P}^x \left( \sup_{s \leq 2t} |X_s-x| \geq 2r \right)
		\leq 3 \sup_{|z-x| \leq r} \mbb{P}^z \left( \sup_{s \leq t} |X_s-z| \geq r \right), \qquad x \in \mbb{R}^d,\;r>0,\;t>0. \label{up-eq4}
	\end{equation}
	To prove this, we note that \begin{equation*}
		\mbb{P}^x \left( \sup_{s \leq 2t} |X_s-x| \geq 2r \right)
		\leq \mbb{P}^x \left( \sup_{s \leq t} |X_s-x| \geq 2r \right) + \mbb{P}^x \left( \sup_{s \leq t} |X_{s+t}-x| \geq 2r \right),
	\end{equation*}
	and, by the Markov property,
	\begin{align*}
		\mbb{P}^x \left( \sup_{s \leq t} |X_{s+t}-x| \geq 2r \right)
		&= \mbb{E}^x \left( \mbb{P}^z \left[ \sup_{s \leq t} |X_s-x| \geq 2r \right] \bigg|_{z=X_t} \right) \\
		&\leq \mbb{P}^x(|X_t-x| \geq r) + \sup_{|z-x| \leq r} \mbb{P}^z \left( \sup_{s \leq t} |X_s-z| \geq r \right).
	\end{align*}
	\secundolist This part of the proof uses an idea from Khintchine \cite{khin39}. Fix $x \in \mbb{R}^d$ such that \eqref{up-eq3} holds. Since $f$ is monotone, we have
	\begin{equation*}
		p_n:=\mbb{P}^x \left( \sup_{2^{-(n+1)} \leq s \leq 2^{-n}} \frac{1}{f(s)} \sup_{r \leq s} |X_r-x| \geq 4 \right)
		\leq \mbb{P}^x \left( \sup_{s \leq 2^{-n}} |X_s-x| \geq 4 f(2^{-(n+1)}) \right)
	\end{equation*}
	for every $n \in \mbb{N}$. Take any $\theta_n \in [2^{-n},2^{-(n-1)}]$, then  $\theta_n/2 \leq 2^{-n}$ and using \primo\, and the monotonicity of $f$, we get
	\begin{align*}
		p_n 
		&\leq  \mbb{P}^x \left( \sup_{s \leq \theta_n} |X_s-x| \geq 4 f(\theta_n/4) \right) 
		\leq 9 \sup_{|z-x| \leq 3f(\theta_n/4)} \mbb{P}^z \left( \sup_{s \leq \theta_n/4} |X_r-z| \geq  f(\theta_n/4) \right).
	\end{align*}
	Writing $\theta_n = 2^{-u}$ for $u \in [n-1,n]$ and integrating with respect to $u \in [n-1,n]$, it follows that
	\begin{align*}
		p_n
		\leq 9 \int_{n-1}^n \sup_{|z-x| \leq 3f(2^{-u-2})} \mbb{P}^z \left( \sup_{s \leq 2^{-u-2}} |X_r-z| \geq  f(2^{-u-2}) \right) \, du.
	\end{align*}
	By a change of variables ($t=2^{-u-2}$), 
	\begin{align*}
		p_n
		\leq \frac{9}{|\log 2|} \int_{2^{-(n+2)}}^{2^{-(n+1)}} \frac{1}{t} \sup_{|z-x| \leq 3 f(t)} \mbb{P}^z \left( \sup_{r \leq t} |X_r-z| \geq f(t) \right) \, dt,
	\end{align*}
	and so \eqref{up-eq3} yields $\sum_{n \in \mbb{N}} p_n < \infty$. Applying the Borel--Cantelli lemma, we conclude that
	\begin{equation*}
		\limsup_{n \to \infty} \sup_{2^{-(n+1)} \leq s \leq 2^{-n}} \frac{1}{f(s)} \sup_{r \leq s} |X_r-x| \leq 4 \quad \text{$\mbb{P}^x$-a.s.} \qedhere
	\end{equation*}
\end{proof}

It is natural to ask whether the two suprema in \eqref{up-eq3} are needed, i.e.\ if upper functions can also be characterized in terms of the integral $\int_0^1 \frac{1}{t} \mbb{P}^x(|X_t-x| \geq C f(t)) \, dt$. Our next result shows that this is possible under some additional assumptions.

\begin{proposition} \label{up-1}
	Let $(X_t)_{t \geq 0}$ be a strong Markov process with \cadlag sample paths. Let $f:[0,1] \to [0,\infty)$ be a non-decreasing function such that\footnote{Here, essinf denotes the essential infimum with respect to Lebesgue measure.}
	\begin{equation}
		C := \essinf \left\{ \limsup_{n \to \infty} \frac{f(s^n)}{f(s^{n+1})}; s \in (0,1) \right\}<\infty. \label{up-eq5}
	\end{equation}
	Assume that the following conditions are satisfied for some constants $\varrho,\kappa>0$ and a function $R:[0,1] \to (0,\infty]$: \begin{align}
		\limsup_{t \to 0} \sup_{|z-x| \leq R(t)} \mbb{P}^z(|X_t-z| \geq \varrho f(t)) &< 1
		\label{up-eq6} \\
		\sum_{n \geq 1} \mbb{P}^x \left( \sup_{u \leq s^n} |X_u-x| > R(s^n) \right) &< \infty \quad \text{for a.e.\ $s \in (0,1)$} \label{up-eq65} \\
		\int_0^1\frac{1}{t} \mbb{P}^x(|X_t-x|> \kappa f(t)) \, dt &< \infty.
		\label{up-eq7}
	\end{align}
	Then
	\begin{equation*}
		\limsup_{t \to 0} \frac{1}{f(t)} \sup_{s \leq t} |X_s-x| \leq C(\varrho+\kappa) \quad \text{$\mbb{P}^x$-a.s.}
	\end{equation*}
\end{proposition}

\begin{remark} \label{up-2} \begin{enumerate}[wide, labelwidth=!, labelindent=0pt]
	\item\label{up-2-i}  Since $f$ is non-decreasing, the constant $C$ in \eqref{up-eq5} is greater or equal than $1$. If $f$ is regularly varying at zero, i.e.\ if the limit
	\begin{equation*}
		L(a) := \lim_{t \to 0} \frac{f(at)}{f(t)}
	\end{equation*}
	exists for all $a>0$, then $C=1$; this follows from the fact that, by Karamata's characterization theorem, see e.g.\ \cite{bingham}, the limit $L$ is of the form $L(a)=a^{\varrho}$ for some $\varrho \geq 0$.
	\item\label{up-2-ii} There is a trade-off between \eqref{up-eq6} and \eqref{up-eq65} regarding the choice of $R$; e.g.\ for $R\equiv \infty$, condition \eqref{up-eq65} is trivially satisfied but a uniform bound for $z \in \mbb{R}^d$ is needed in \eqref{up-eq6}. 
	\item If $(X_t)_{t \geq 0}$ is a L\'evy-type process, then the maximal inequality \eqref{max-eq0} shows that \eqref{up-eq65} is automatically satisfied for $R(t)\equiv R$ constant.
	\item\label{up-2-iii} It is not hard to check that 
	\begin{equation}
		\int_{(0,1)} \frac{1}{t} \mbb{P}^x \left( \sup_{s \leq t} |X_s-x| > \kappa f(t) \right) \, dt < \infty \label{up-eq8}
	\end{equation}
	implies that \eqref{up-eq65} and \eqref{up-eq7} hold with $R(t)=\kappa f(t)$.
\end{enumerate}
\end{remark}

For the proof of Proposition~\ref{up-1} we use the following Ottaviani-type inequality; for $R=\infty$ this is the classical Ottaviani inequality for Markov processes, see e.g.\ \cite[p.~420]{gikh1} or \cite[p.~125]{ito}.

\begin{lemma} \label{up-3}
	Let $(X_t)_{t \geq 0}$ be a strong Markov process with \cadlag sample paths. Then
	\begin{equation*}
		\mbb{P}^x \left( \sup_{s \leq t} |X_s-x| > u+v \right)
		\leq \frac{1}{1-\alpha_{R,x}(t,u)} \left[ \mbb{P}^x(|X_t-x| > v) + \mbb{P}^x \left( \sup_{s \leq t} |X_s-x| >R \right) \right]
	\end{equation*}
	for all $x \in \mbb{R}^d$, $u,v>0$ and $R \in (0,\infty]$, where
	\begin{equation*}
		\alpha_{R,x}(t,u) := \sup_{s \leq t} \sup_{|z-x| \leq R} \mbb{P}^z(|X_s-z| \geq u).
	\end{equation*}
\end{lemma}

\begin{proof}
	Denote by $\tau_r^x$ the first exit time of $(X_t)_{t \geq 0}$ from the closed ball $\overline{B(x,r)}$ and set $\sigma := \tau_{u+v}^x$ for fixed $u,v>0$. We have
	\begin{align*}
		\mbb{P}^x \left( \sup_{s \leq t} |X_s-x| > u + v \right)
		&= \mbb{P}^x(\sigma \leq t) \\
		&\leq \mbb{P}^x(|X_t-x| > v) + \mbb{P}^x \left( |X_t-x| \leq v, \sigma \leq t, \tau_R^x>t \right) + \mbb{P}^x(\tau_R^x \leq t).
	\end{align*}
	By the strong Markov property,
	\begin{align*}
		\mbb{P}^x \left( |X_t-x| \leq v, \sigma \leq t, \tau_R^x>t \right) 
		&\leq \mbb{P}^x \left( |X_t-X_{\sigma}| \geq u, \sigma \leq t, |X_{\sigma}-x| \leq R \right)  \\
		&= \mbb{E}^x \left[ \I_{\{\sigma \leq t\}} \I_{\{|X_{\sigma}-x| \leq R\}} \mbb{P}^z(|X_{t-s}-z| \geq  u) \big|_{z=X_{\sigma},s=\sigma} \right] \\
		&\leq \alpha_{R,x}(t,u) \mbb{P}^x(\sigma \leq t),
	\end{align*}
	and so
	\begin{align*}
		\mbb{P}^x \left( \sup_{s \leq t} |X_s-x| > u + v \right) (1-\alpha_{R,x}(t,u)) 
		&\leq \mbb{P}^x(|X_t-x| > v) + \mbb{P}^x(\tau_R^x \leq t) \\
		&= \mbb{P}^x(|X_t-x| > v) + \mbb{P}^x \left( \sup_{s \leq t} |X_s-x|>R \right). \qedhere
	\end{align*}
\end{proof}

\begin{proof}[Proof of Proposition~\ref{up-1}] 
	\primolist Claim: \begin{equation}
		\sum_{n \in \mbb{N}} \int_0^1 \mbb{P}^x(|X_{s^n}-x| > \kappa f(s^n)) \log \frac{1}{s} \, ds < \infty.
		\label{up-st1}
	\end{equation}
	By a change of variables, $t=s^n$, $dt=n t^{(n-1)/n} \, ds$, we find that
	\begin{align*}
		\int_0^1 \mbb{P}(|X_{s^n}-x| > \kappa f(s^n)) \log \frac{1}{s} \, ds 
		= \int_0^1 \frac{1}{n^2} t^{1/n} \log \frac{1}{t} \mbb{P}^x(|X_t-x|> \kappa f(t)) \frac{1}{t} \, dt.
	\end{align*}
	As
	\begin{equation*}
		\sum_{n \in \mbb{N}}  \frac{1}{n^2} t^{1/n} \log \frac{1}{t} \leq 2, \qquad t \in (0,1),
	\end{equation*}
	cf.\ Lemma~\ref{appendix-2}, the monotone convergence theorem yields
	\begin{equation*}
		\sum_{n \in \mbb{N}} 	\int_0^1 \mbb{P}(|X_{s^n}-x| > \kappa f(s^n)) \log \frac{1}{s} \, ds 
		\leq 2 \int_0^1 \mbb{P}^x(|X_t-x|> \kappa f(t)) \frac{1}{t} \, dt,
	\end{equation*}
	and the latter integral is finite by \eqref{up-eq7}. This proves \eqref{up-st1}. In particular, there is a Lebesgue null set $N \subseteq (0,1)$ such that
	\begin{equation}
		\sum_{n \in \mbb{N}} \mbb{P}^x(|X_{s^n}-x|> \kappa f(s^n)) < \infty \fa s \in (0,1) \setminus N.
		\label{up-st2}
	\end{equation}
	\secundolist Fix $\eps>0$, and take $s \in (0,1) \setminus N$ such that $\limsup_{n \to \infty} f(s^n)/f(s^{n+1}) \leq (C+\eps)$ for the constant $C$ defined in \eqref{up-eq5}. By Lemma~\ref{up-3}, we have
	\begin{align*}
		&\mbb{P}^x \left( \sup_{r \leq s^n} |X_r-x| > (\kappa+\varrho) f(s^n) \right) \\
		&\quad\leq \frac{1}{1-\alpha_{R,x}(s^n,\varrho f(s^n))} \left[\mbb{P}^x(|X_{s^n}-x| > \kappa f(s^n)) + \mbb{P}^x \left( \sup_{u \leq s^n} |X_u-x| > R(s^n) \right) \right],
	\end{align*}
	where $\alpha_{R,x}(t,r) := \sup_{u \leq t} \sup_{|z-x| \leq R(t)} \mbb{P}^z(|X_u-z| \geq r)$. From \eqref{up-eq6} and the monotonicity of $f$, we see that there exists some $\delta \in (0,1)$ such that
	\begin{equation}
		\alpha_{R,x}(s^n,\varrho f(s^n))
		\leq \sup_{r \leq s^n} \sup_{|z-x| \leq R(s^n)} \mbb{P}^z(|X_{r}-z| \geq \varrho f(r))
		\leq 1-\delta \label{up-eq85}
	\end{equation}
	for $n \gg 1$. Thus, \begin{align*}
		\mbb{P}^x \left( \sup_{r \leq s^n} |X_r-x| > (\kappa+\varrho) f(s^n) \right) 
		\leq \frac{1}{\delta} \left[\mbb{P}^x(|X_{s^n}-x| > \kappa f(s^n)) + \mbb{P}^x \left( \sup_{u \leq s^n} |X_u-x| > R(s^n) \right) \right],
	\end{align*}
	which implies, by \eqref{up-st2} and \eqref{up-eq65},
	\begin{equation*}
		\sum_{n \in \mbb{N}} \mbb{P}^x \left( \sup_{r \leq s^n} |X_r-x| > (\kappa+\varrho) f(s^n) \right)  < \infty.
	\end{equation*}
	Applying the Borel-Cantelli lemma gives
	\begin{equation*}
		\limsup_{n \to \infty} \frac{1}{f(s^n)} \sup_{r \leq s^n} |X_r-x| \leq \varrho+\kappa \quad \text{$\mbb{P}^x$-a.s.}
	\end{equation*}
	If $t \in [s^{n+1},s^n)$ for some $n \gg 1$, then
	\begin{equation*}
		\frac{1}{f(t)} \sup_{r \leq t} |X_r-x|
		\leq \frac{1}{f(s^{n+1})} \sup_{r \leq s^n} |X_r-x|
		= \frac{f(s^n)}{f(s^{n+1})} \frac{1}{f(s^n)} \sup_{r \leq s^n} |X_r-x|,
	\end{equation*}
	and so
	\begin{equation*}
		\limsup_{t \to 0} \frac{1}{f(t)} \sup_{r \leq t} |X_r-x|
		\leq \limsup_{n \to \infty} \left( \frac{f(s^n)}{f(s^{n+1})} \frac{1}{f(s^n)} \sup_{r \leq s^n} |X_r-x| \right)
		\leq (C+\eps)(\kappa+\varrho)
	\end{equation*}
	$\mbb{P}^x$-almost surely. Since $\eps>0$ is arbitrary, this finishes the proof.
\end{proof}

Combining Theorem~\ref{up-4} with the maximal inequality \eqref{max-eq0}, we get the following criterion; see \cite[Proposition 9]{knop14} for a closely related result.

\begin{corollary} \label{up-9}
	Let $(X_t)_{t \geq 0}$ be a L\'evy-type process with symbol $q$. If $f:[0,1] \to [0,\infty)$ is a non-decreasing function such that
	\begin{equation}
		\int_0^1 \sup_{|z-x| \leq f(t)} \sup_{|\xi| \leq 1/(C f(t))} |q(z,\xi)| \, dt < \infty \label{up-eq25}
	\end{equation}
	for some constant $C>0$, then
	\begin{equation*}
		\lim_{t \to 0} \frac{1}{f(t)} \sup_{s \leq t} |X_s-x| = 0 \quad \text{$\mbb{P}^x$-a.s.}
	\end{equation*}
\end{corollary}

\begin{proof}
	If the integral in \eqref{up-eq25} is finite some $C>0$, then it is finite for arbitrary small $C>0$. \emph{Indeed:} Since $\xi \mapsto \sqrt{|q(z,\xi)|}$ is subadditive, we have
	\begin{align*}
		|q(z,2\xi)| =|q(z,\xi+\xi)| 
		\leq \left(\sqrt{|q(z,\xi)|}+\sqrt{|q(z,\xi)|}\right)^2 
		= 4 |q(z,\xi)|,
	\end{align*}
	which implies that
	\begin{align*}
		\int_0^1 \sup_{|z-x| \leq f(t)} \sup_{|\xi| \leq 1/(2^{-n} C f(t))} |q(z,\xi)| \, dt
		\leq 4^n \int_0^1 \sup_{|z-x| \leq f(t)} \sup_{|\xi| \leq 1/(C f(t))} |q(z,\xi)| \, dt < \infty
	\end{align*}
	for every $n \in \mbb{N}$. Applying the maximal inequality \eqref{max-eq0} and Theorem~\ref{up-4} yields
	\begin{align*}
		\limsup_{t \to 0} \frac{1}{f(t)} \sup_{s \leq t} |X_s-x| \leq 4 C 2^{-n} \quad \text{$\mbb{P}^x$-a.s.}
	\end{align*}
	Letting $n \to \infty$ proves the assertion.
\end{proof}

We conclude this section with the following result on the growth of sample paths of L\'evy-type processes.

\begin{proposition} \label{up-11}
	Let $(X_t)_{t \geq 0}$ be a L\'evy-type process with symbol $q$. Then: \begin{enumerate}
		\item\label{up-11-i} $\limsup_{t \to 0} t^{-\kappa} \sup_{s \leq t} |X_s-x| = 0$ $\mbb{P}$-a.s.\ for any $\kappa < \tfrac{1}{2}$.
		\item\label{up-11-ii} If $x \in \mbb{R}^d$ is such that \begin{equation}
			\sup_{|z-x| \leq R} \sup_{|\xi| \leq r} |q(z,\xi)| \leq c \frac{r^2}{|\log r|^{1+\eps}}, \qquad r \gg 1, \label{up-eq31}
		\end{equation}
		for some constants $R>0$, $c>0$ and $\eps>0$, then
		\begin{equation*}
			\limsup_{t \to 0} \frac{1}{\sqrt{t \log |\log t|}} \sup_{s \leq t} |X_s-x| = 0 \quad \text{$\mbb{P}^x$-a.s.}
		\end{equation*}
	\end{enumerate}
\end{proposition}

Khintchine \cite{khin39} (see also \cite[Appendix, Theorem 4]{skorohod91}) showed that any L\'evy process without Gaussian component satisfies
\begin{equation*}
	\limsup_{t \to 0} \frac{|X_t|}{\sqrt{t \log |\log t|}} = 0 \quad \text{a.s.}
\end{equation*}
One might expect that an analogous result holds for L\'evy-type processes but this does not seem to follow from our results; note that \eqref{up-eq31} is stronger than assuming that $(X_t)_{t \geq 0}$ has no Gaussian component, cf.\ \cite[Lemma A.3]{ihke}.

\begin{proof}[Proof of Proposition~\ref{up-11}]
	\primolist Because of the subadditivity of $\xi \mapsto \sqrt{|q(x,\xi|}$, it holds that
	\begin{equation*}
		|q(x,\xi)| \leq \sup_{|\eta| \leq 1} |q(x,\eta)| (1+|\xi|^2),
	\end{equation*}
	cf.\ \cite[Theorem 6.2]{barca}, and so 
	\begin{equation*}
		\sup_{|z-x| \leq 1} \sup_{|\xi| \leq r} |q(z,\xi)| \leq c' (1+r^2)
	\end{equation*}
	for some constant $c'>0$. Hence,
	\begin{equation*}
		\int_0^1 \sup_{|z-x| \leq 1} \sup_{|\xi| \leq 1/(C t^{\kappa})} |q(z,\xi)| \, dt < \infty
	\end{equation*}
	for any $\kappa \in (0,\tfrac{1}{2})$ and $C>0$. By Corollary~\ref{up-9}, this proves \ref{up-11-i}.
	\secundolist Set $f(t) := \sqrt{t \log \log \frac{1}{t}}$, then, by \eqref{up-eq31},
	\begin{equation*}
		\int_0^{1/e^e} \sup_{|z-x| \leq R} \sup_{|\xi| \leq 1/(C f(t))} |q(z,\xi)| \, dt
		\leq \frac{c}{C^2} \int_0^{1/e^e} \frac{1}{t \log \log \frac{1}{t}} \frac{1}{|\log \sqrt{C^2 t \log \log \frac{1}{t}}|^{1+\eps}} \, dt
	\end{equation*}
	for every $C>0$, and the latter integral is finite. Corollary~\ref{up-9} gives the assertion.
\end{proof}

In the remainder of the article, we prove the results announced in Section~\ref{main}.

\section{Proofs of Theorem~\ref{main-3} and Theorem~\ref{main-5}} \label{p}

For the proof of Theorem~\ref{main-3} and Theorem~\ref{main-5}, we need the following result which links two of our integral conditions.

\begin{lemma} \label{p-1}
	Let $\psi: \mbb{R}^d \to \mbb{C}$ be a continuous negative definite function with L\'evy triplet $(b,0,\nu)$, and set
	\begin{equation*}
		\psi^*(r) := \sup_{|\xi| \leq r} \re \psi(\xi), \qquad r>0.
	\end{equation*}
	If $f:[0,1] \to [0,\infty)$ is a non-decreasing function, then the implication
	\begin{equation*}
		\int_0^1 \nu(\{|y| \geq f(t)\}) \, dt < \infty \implies \int_0^1 \psi^* \left( \frac{1}{f(t)} \right) \, dt < \infty
	\end{equation*}
	holds in each of the following two cases.
	\begin{enumerate}[label*=\upshape (A\arabic*),ref=\upshape A\arabic*] 
		\item The L\'evy measure $\nu$ satisfies \begin{equation*}
				\limsup_{r \to 0} \frac{\int_{|y| \leq r} |y|^2 \, \nu(dy)}{r^2 \nu(\{|y| > r\})} < \infty. 
			\end{equation*}
		\item There is a constant $c>0$ such that \begin{equation*}
			\int_{r<f(t)} \frac{1}{f(t)^2} \, dt \leq c \frac{f^{-1}(r)}{r^2}, \qquad r \in (0,1).
		\end{equation*}
	\end{enumerate}
\end{lemma}

\begin{remark} \label{p-3} 
\begin{enumerate}[wide, labelwidth=!, labelindent=0pt]
	\item\label{p-3-i} There are several equivalent formulations of condition \eqref{A1} in terms of so-called concentration functions. If we define, following \cite{pruitt81},
	\begin{equation*}
		G(r) := \nu(\{|y| > r\}) \quad \text{and} \quad K(r) := \frac{1}{r^2} \int_{|y| \leq r} |y|^2 \, \nu(dy),
	\end{equation*}
	then \eqref{A1} can be stated equivalently in the following way:
	\begin{equation*}
		\limsup_{r \to 0} \frac{K(r)}{G(r)}<\infty.
	\end{equation*}
	Set
	\begin{equation*}
		h(r) := \int_{y \neq 0} \min\left\{1,\frac{|y|^2}{r^2}\right\} \, \nu(dy) = K(r)+G(r),
	\end{equation*}
	then we see that
	\begin{equation}
		\eqref{A1} \iff \liminf_{r \to 0} \frac{G(r)}{h(r)} >0. \label{p-eq4}
	\end{equation}
	Since 
	\begin{equation*}
		\frac{1}{c} h(r) \leq \psi^* \left( \frac{1}{r} \right) \leq c h(r), \qquad r>0,
	\end{equation*}
	for some constant $c>0$, depending only on the dimension $d$, see e.g.\ \cite[Lemma 5.1 and p.~595]{rs-growth} or \cite[Lemma 4]{grz}, it follows that
	\begin{equation}
		\eqref{A1} \iff \liminf_{r \to 0} \frac{G(r)}{\psi^*(1/r)} = \liminf_{r \to 0} \frac{\nu(\{|y| > r\})}{\psi^*(1/r)}>0. \label{p-eq5}
	\end{equation}
	Moreover, there is a sufficient condition for \eqref{A1} in terms of the function 
	\begin{equation*}
		I(r) := \int_{y \neq 0} \min\{r^2,|y|^2\} \, \nu(dy)=r^2 h(r);
	\end{equation*}
	namely, if
	\begin{equation}
		 \liminf_{r \to 0} \frac{I(2r)}{I(r)}>1, \label{p-eq7}
	\end{equation}
	then \eqref{A1} holds. \emph{Indeed:} By Tonelli's theorem,
	\begin{align*}
		I(r) 
		= \int_{y \neq 0} \int_0^{\min\{r^2,|y|^2\}} \, dz \, \nu(dy)
		&= \int_{\mbb{R}} \int_{y \neq 0} \I_{\{|z| < r^2\}} \I_{\{|z| < |y|^2\}} \, \nu(dy) \, dz  \\
		&= \int_0^{r^2} \nu(\{|y| > \sqrt{z}\}) \, dz.
	\end{align*}
	Thus, 
	\begin{equation*}
		I(2r) = I(r) + \int_{r^2}^{4r^2} \nu(\{|y| > \sqrt{z}\}) \, dz \leq I(r) + 3r^2 \nu(\{ |y| > r\}).
	\end{equation*}
	Consequently, \eqref{p-eq7} implies that
	\begin{equation*}
		1 < \liminf_{r \to 0} \frac{I(2r)}{I(r)} \leq 1+ \liminf_{r \to 0} \frac{3r^2 \nu(\{|y| > r\})}{I(r)}.
	\end{equation*}
	As $I(r)=r^2 h(r)$, this is equivalent to \eqref{p-eq4} and hence to \eqref{A1}. Let us mention that a condition similar to \eqref{p-eq7} appears in the monograph \cite{bertoin} by Bertoin in the study of upper functions for sample paths of subordinators.
	\item\label{p-3-ii} If $\nu$ is the L\'evy measure of a one-dimensional L\'evy process and $\nu(\{|y| \geq r\})$ grows faster than $\log r$ as $r \to 0$, then \eqref{A1} implies that (the law of) $X_t$ has a smooth density $p_t \in C_b^{\infty}(\mbb{R})$ for every $t>0$, see \cite[Section 5]{kallenberg} and also \cite[p.~127]{knop13}.
\end{enumerate}
\end{remark}

\begin{proof}[Proof of Lemma~\ref{p-1}]
	\primolist Suppose that \eqref{A1} holds. Then there is some constant $c>0$ such that
	\begin{equation*}
		\psi^* \left( \frac{1}{r} \right) 
		= \sup_{|\xi| \leq 1/r} \re \psi(\xi)
		\leq c \nu(\{|y| > r\})
	\end{equation*}
	for $r>0$ small, cf.\ \eqref{p-eq5}. Since we may assume without loss of generality that $f(t) \to 0$ as $t \downarrow 0$, we find that
	\begin{equation*}
		\int_0^{\delta} \psi^* \left( \frac{1}{f(t)} \right) \, dt \leq c \int_0^{\delta} \nu(\{|y| > f(t)\}) \, dt
	\end{equation*}
	for some $\delta>0$. As $\psi$ is bounded on compact sets, this proves the assertion.
	\secundolist Suppose that \eqref{A2} holds. From \begin{equation*}
		\psi^*(r) = \sup_{|\xi| \leq r} \re \psi(\xi) \leq 2 \int_{y \neq 0} \min\{1,|y|^2 r^2\} \, \nu(dy),
	\end{equation*}
	we get
	\begin{align}
		\int_0^1 \psi^* \left( \frac{1}{f(t)} \right) \, dt
		&\leq 2 \int_0^1 \frac{1}{f(t)^2} \int_{|y| \leq f(t)} y^2 \, \nu(dy) \, dt + 2 \int_0^1 \nu(\{|y| > f(t)\}) \, dt. \label{p-eq9} 
	\end{align}
	The second integral on the right-hand side of \eqref{p-eq9} is finite by assumption, and so it suffices to show that the first integral
	\begin{equation*}
		J := \int_0^1 \frac{1}{f(t)^2} \int_{|y| \leq f(t)} y^2 \, \nu(dy) \, dt
	\end{equation*}
	is finite. By Tonelli's theorem and \eqref{A2}, we have
	\begin{align*}
		J 
		= \int_{y \neq 0} |y|^2 \int_{f(t) \geq |y|} \frac{1}{f(t)^2} \, dt \, \nu(dy)
		\leq c \int_{y \neq 0} f^{-1}(|y|) \, \nu(dy).
	\end{align*}
	Since $f$ is non-decreasing, we find by another application of Tonelli's theorem that
	\begin{align*}
		J 
		\leq c \int_{y \neq 0} \int_{t \leq f^{-1}(|y|)} \, dt \, \nu(dy) 
		&\leq c \int_{y \neq 0} \int_{f(t) \leq |y|} \, dt \, \nu(dy)
		= c \int_{0}^1 \nu(\{|y| \geq f(t)\}) \, dt<\infty. \qedhere
	\end{align*}
\end{proof}

\begin{proof}[Proof of Theorem~\ref{main-3}]
	\ref{main-3-i} $\implies$ \ref{main-3-ii}: If $\int_0^1 \nu(\{|y| \geq  c f(t)\}) \, dt < \infty$, then it follows from Lemma~\ref{p-1} and the sector condition that
	\begin{equation*}
		\int_0^1 \sup_{|\xi| \leq 1/(c f(t))} |\psi(\xi)| \, dt < \infty.
	\end{equation*}
	By the subadditivity of $\xi \mapsto \sqrt{|\psi(\xi)|}$, this implies that
	\begin{equation*}
		\int_0^1 \sup_{|\xi| \leq 1/(2^{-n} c f(t))} |\psi(\xi)| \, dt < \infty.
	\end{equation*}
	for all $n \in \mbb{N}$, see the proof of Corollary~\ref{up-9}. Since the integral expression is monotone w.r.t.\ $c$, we conclude that
	\begin{equation*}
		\int_0^1 \sup_{|\xi| \leq 1/( c f(t))} |\psi(\xi)| \, dt < \infty \fa c>0.
	\end{equation*}
	\ref{main-3-ii} $\implies$ \ref{main-3-iii}: This is clear from the maximal inequality, cf.\ Proposition~\ref{max-0}. \par
	\ref{main-3-iii}$\iff$\ref{main-3-iv}: The implication $\ref{main-3-iii}\implies\ref{main-3-iv}$ is obvious. The other direction is immediate from Etemadi's inequality, see e.g.\ \cite[Theorem 22.5]{billingsley} or \cite[Theorem 7.6]{gut}, which shows that
	\begin{equation*}
		\mbb{P} \left( \sup_{s \leq t} |X_s| \geq 3r \right) \leq 3 \mbb{P}(|X_t| \geq r), \qquad r>0,\;t>0.
	\end{equation*}
	\ref{main-3-iii} $\implies$ \ref{main-3-v}: This is immediate from Theorem~\ref{up-4}; note that the supremum in \eqref{up-eq3} breaks down because L\'evy processes are homogenous in space. \par
	\ref{main-3-v}$\implies$\ref{main-3-vi}$\implies$\ref{main-3-vii}: Obvious. \par
	\ref{main-3-vii}$\implies$\ref{main-3-i}: In dimension $d=1$, this follows from \cite[Proposition 4.2]{bertoin08}. The following reasoning works in any dimension $d \geq 1$. By Blumenthal's 0-1 law, there exists a constant $C>0$ such that
	\begin{equation}
		\limsup_{t \to 0} \frac{1}{f(t)} |X_t| \leq \frac{C}{2} \quad \text{almost surely.} \label{p-eq13}
	\end{equation}
	Suppose that $\int_0^1 \nu(\{|y| \geq 2C f(t)\}) \, dt$ is infinite. As $f$ is non-decreasing, the series test yields
	\begin{equation}
		\sum_{n =2}^{\infty} \nu(\{|y| \geq 2C f(1/n)\}) \left( \frac{1}{n-1}-\frac{1}{n}\right)=\infty. \label{p-eq14}
	\end{equation}
	The random variables
	\begin{equation*}
		N_{s,t}^{(r)} := \sharp \{u \in (s,t]; |\Delta X_u| \geq r\}, \qquad 0\leq s<t, \, r>0,
	\end{equation*}
	are Poisson distributed with parameter $(t-s) \nu(\{|y| \geq r\})$, and so $Y_n := N_{1/(n+1),1/n}^{2Cf(1/n)}$ are Poisson distributed with parameter $\lambda_n := \nu(\{|y| \geq 2Cf(1/n)\}) \left( \frac{1}{n}-\frac{1}{n+1} \right)$. Using the elementary estimate $1-e^{-x} \geq x/(1+x)$, we get
	\begin{equation*}
		\sum_{n \in \mbb{N}} \mbb{P}(Y_n \geq 1)
		= \sum_{n \in \mbb{N}} \left( 1- e^{-\lambda_n} \right)
		\geq \sum_{n \in \mbb{N}} \frac{\lambda_n}{1+\lambda_n} 
		\geq \sum_{n \in \mbb{N}} \min\left\{\lambda_n,\frac{1}{2} \right\}
		= \infty;
	\end{equation*}
	here we use that \eqref{p-eq14} implies $\sum_{n \in \mbb{N}} \lambda_n = \infty$ because $\frac{1}{n-1}-\frac{1}{n} \approx \frac{1}{n^2} \approx \frac{1}{n+1}-\frac{1}{n}$ for $n \gg 1$. 
	Since the random variables $Y_n$, $n \in \mbb{N}$, are independent, the Borel--Cantelli lemma shows that the event $\{Y_n \geq 1$ infinitely often$\}$ has probability $1$. Thus, with probability $1$ there are infinitely many $n \in \mbb{N}$ such that $|\Delta X_u| \geq 2C f(1/n)$ for some $u \in [\frac{1}{n+1},\frac{1}{n}]$. Since either $|X_u| \geq C f(1/n) \geq Cf(u)$ or $|X_{u-}| \geq Cf(1/n) \geq Cf(u-)$ for any such $u \in [\frac{1}{n+1},\frac{1}{n}]$, we conclude that
	\begin{equation*}
		\limsup_{t \to 0} \frac{1}{f(t)} |X_t| \geq C \quad \text{almost surely,}
	\end{equation*}
	which contradicts \eqref{p-eq13}. Hence, $\int_0^1 \nu(\{|y| \geq 2Cf(t)\}) \, dt < \infty$. See (the proof of) Theorem~\ref{main-9} for an alternative reasoning. \par
	The random variables $\limsup_{t \to 0} \frac{1}{f(t)} |X_t|$ and $\limsup_{t \to 0} \frac{1}{f(t)} \sup_{s \leq t} |X_s|$ are $\mathcal{F}_{0+}$-measurable, and therefore Blumenthal's 0-1-law shows that the events in \ref{main-3-v}-\ref{main-3-vii} have probability $0$ or $1$. Consequently, 'almost surely' may be replaced by 'with positive probability' in each of the statements.
\end{proof}

\begin{proof}[Proof of Theorem~\ref{main-5}] 
	\ref{main-5-i}$\implies$\ref{main-5-ii}: Without loss of generality, $f(t) \to 0=f(0)$ as $t \downarrow 0$; otherwise the assertion is immediate from the local boundedness of the symbol, cf.\ \eqref{def-eq13}. It follows from \eqref{A1'} that
	\begin{equation*}
		\liminf_{r \to 0} \inf_{|z-x| \leq R} \frac{\nu(z,\{|y|>r\})}{\sup_{|\xi| \leq 1/r} \re q(z,\xi)} >0,
	\end{equation*}
	see Remark~\ref{p-3}\ref{p-3-i}. Since the sector condition holds (with a constant not depending on $z \in \overline{B(x,R)}$), we find that
	\begin{equation*}
			\liminf_{r \to 0} \inf_{|z-x| \leq R} \frac{\nu(z,\{|y|>r\})}{\sup_{|\xi| \leq 1/r} |q(z,\xi)|} >0,
	\end{equation*}
	i.e.\ there are constants $K>0$ and $\delta>0$ such that
	\begin{equation*}
		\sup_{|\xi| \leq 1/r} |q(z,\xi)| \leq K \nu(z,\{|y|>r\}), \qquad z \in \overline{B(x,R)},
	\end{equation*}
	for $r \leq \delta$. As $f(t) \to 0$ as $t \downarrow 0$, this implies
	\begin{equation*}
		\sup_{|z-x| \leq f(t)} \sup_{|\xi| \leq 1/(cf(t))} |q(z,\xi)|
		\leq K \sup_{|z-x| \leq f(t)} \nu(z,\{|y| > c f(t)\})
	\end{equation*}
	for $t>0$ small. Integrating with respect to $t$ and using the local boundedness of $q$, we conclude that
	\begin{equation*}
		\int_0^1 \sup_{|z-x| \leq f(t)} \sup_{|\xi| \leq 1/(cf(t))} |q(z,\xi)| \, dt < \infty.
	\end{equation*}
	\ref{main-5-ii}$\implies$\ref{main-5-iii}: If the integral in \ref{main-5-ii} is finite for some $\eps>0$, then it is finite for all $\eps>0$; this follows from the subadditivity of $\xi \mapsto \sqrt{|q(z,\xi)|}$, see the proof of Corollary~\ref{up-9}. The implication \ref{main-5-ii}$\implies$\ref{main-5-iii} is now immediate from the maximal inequality \eqref{max-eq0}. \par
	\ref{main-5-iii}$\implies$\ref{main-5-iv}: cf.\ Theorem~\ref{up-4}.
\end{proof}

\section{Proof of the converse and the lower growth bounds} \label{conv}

In this section, we present the proofs of Proposition~\ref{main-7}, Proposition~\ref{main-8} and Theorem~\ref{main-9}.

\begin{proof}[Proof of Proposition~\ref{main-7}] 
	\primolist 	If	
	\begin{equation*}
		\int_0^1 \sup_{|z-x| \leq f(t)} \mbb{P}^z \left( \sup_{s \leq t} |X_s-z| \geq f(t) \right) \frac{1}{t} \, dt < \infty,
	\end{equation*}
	then Theorem~\ref{up-4} shows that $\limsup_{t \to 0} \frac{1}{f(t)} \sup_{s \leq t} |X_s-x|\leq 4$ $\mbb{P}^x$-almost surely. Consequently, $\mbb{P}^x(A_k) \to 1$ for $A_k := \{\forall t \leq 1/k\::\: \frac{1}{f(t)} \sup_{s \leq t} |X_s-x| < 5\}$. Hence,
	\begin{equation*}
		\sup_{t \leq 1/k} \mbb{P}^x \left( \sup_{s \leq t} |X_s-x|\geq 5 f(t)\right) 
		\leq \mbb{P}^x(A_k^c) 
		\xrightarrow[k \to \infty]{} 0.
	\end{equation*}
	 By Corollary~\ref{max-5}, this implies
	 \begin{equation*}
	 	\mbb{P}^x \left( \sup_{s \leq t} |X_s-x|\geq 5 f(t)\right) 
		\geq \frac{1}{2} t G(x,10f(t)), \qquad t \leq \frac{1}{k},
	\end{equation*}
	for $k \gg1$ sufficiently large, where $G(x,r) := \inf_{|z-x| \leq r} \nu(z,\{|y| > r\})$. Dividing both sides by $t$ and integrating over $t \in (0,1)$ yields $\int_0^1 G(x,10f(t)) \, dt < \infty$, which proves \ref{main-7-i}.
	\secundolist Suppose that
	\begin{equation*}
		\int_0^1 G(x,C f(t)) \, dt < \infty
	\end{equation*}
	for some $C>0$ and $G(x,r)$ as in \primo, and assume that \eqref{A1'} holds for some $R>0$. It follows from Remark~\ref{p-3} that there is some constant $\gamma>0$ such that
	\begin{equation*}
		\liminf_{r \to 0} \inf_{|z-x| \leq R} \frac{\nu(z,\{|y| > r\})}{\sup_{|\xi| \leq 1/r} \re q(z,\xi)} \geq \gamma.
	\end{equation*}
	Thus, 
	\begin{equation*}
		\inf_{|z-x| \leq Cf(t)} \sup_{|\xi| \leq 1/(Cf(t))}  \re q(z,\xi) \leq \frac{1}{\gamma} G(x,C f(t))
	\end{equation*}
	for $t>0$ small. Since the symbol $q$ is bounded on compact sets, integration with respect to $t$ gives
	\begin{equation*}
		\int_0^1 \inf_{|z-x| \leq C f(t)} \sup_{|\xi| \leq 1/(Cf(t))}  \re q(z,\xi) \, dt < \infty.
	\end{equation*}
	Because of the subadditivity of the mapping $\xi \mapsto \sqrt{\re q(z,\xi)}$, we may replace $1/(Cf(t))$ by $c/f(t)$ for any $c>0$, compare the proof of Corollary~\ref{up-9}.
\end{proof}

\begin{proof}[Proof of Proposition~\ref{main-8}]
	Let $f:[0,1] \to [0,\infty)$ be such that
	\begin{equation}
		\limsup_{t \to 0} t \inf_{|z-x| \leq  R f(t)} \sup_{|\xi| \leq 1/(Cf(t))} \re q(z,\xi)= \infty, \label{conv-eq9}
	\end{equation}
	for all $R \geq 1$ and some constant $C=C(R)>0$. 
	\primolist Claim: the convergence in \eqref{conv-eq9} holds for any $C>0$. \emph{Indeed}: Clearly, it suffices to show that \eqref{conv-eq9} holds with $C$ replaced by $C 2^n$, $n \in \mbb{N}$. Because of the subadditivity of $\xi \mapsto \sqrt{\re q(z,\xi)}$, we have $\re q(z,2\xi) \leq 4 \re q(z,\xi)$ for all $\xi,z \in \mbb{R}^d$ implying
	\begin{equation*}
		\sup_{|\xi| \leq r} \re q(z,\xi)
		\geq \frac{1}{4} \sup_{|\xi| \leq 2r} \re q(z,\xi)
		\geq \ldots
		\geq \frac{1}{4^n} \sup_{|\xi| \leq 2^n r} \re q(z,\xi)
	\end{equation*}
	for all $r>0$. Using this estimate for $r=1/(2^n C f(t))$, we see that \eqref{conv-eq9} holds with $C$ replaced by $C2^n$.
	\secundolist The idea for this part of the proof is from \cite{rs-growth}. For fixed $R \geq 1$, pick $(t_k)_{k \in \mbb{N}} \subseteq (0,1)$ with $t_k \downarrow 0$ and
	\begin{equation*}
		\lim_{k \to \infty} t_k \inf_{|z-x|  \leq R f(t_k)} \sup_{|\xi| \leq 1/(Rf(t_k))} \re q(z,\xi) = \infty.
	\end{equation*}
	Then the maximal inequality \eqref{max-eq4} shows that
	\begin{equation*}
		\mbb{P}^x \left(  \sup_{s \leq t_k} |X_s-x| < R f(t_k) \right) \xrightarrow[]{k \to \infty} 0,
	\end{equation*}
	and so, by Fatou's lemma,
	\begin{align*}
		\mbb{P}^x \left( \limsup_{k \to \infty} \left\{ \sup_{s \leq t_k} |X_s-x| \geq R f(t_k) \right\} \right)
		&\geq \limsup_{k \to \infty} \mbb{P}^x \left( \sup_{s \leq t_k} |X_s-x| \geq R f(t_k) \right) \\
		&= 1- \liminf_{k \to \infty} \mbb{P}^x \left( \sup_{s \leq t_k} |X_s-x|< R f(t_k) \right) \\
		&= 1.
	\end{align*}
	Consequently, there is a measurable set $\Omega_0$ with $\mbb{P}^x(\Omega_0)=1$ such that $\sup_{s \leq t_k} |X_{t_k}(\omega)-x| \geq R f(t_k)$ infinitely often for every $\omega \in \Omega_0$. In particular,
	\begin{equation*}
		\limsup_{k \to \infty} \frac{1}{f(t_k)} \sup_{s \leq t_k} |X_s(\omega)-x| \geq R, \qquad \omega \in \Omega_0.
	\end{equation*}
	\tertiolist Now assume additionally that $f$ is non-decreasing. For $\omega \in \Omega_0$, let $s_k=s_k(\omega) \in [0,t_k]$ be such that
	\begin{equation*}
		|X_{s_k}(\omega)-x| \geq \frac{1}{2} \sup_{s \leq t_k} |X_s(\omega)-x|.
	\end{equation*}
	By the monotonicity, we have $f(s_k) \leq f(t_k)$, and so 
	\begin{equation*}
		\limsup_{t \to 0} \frac{1}{f(t)} |X_t(\omega)-x|
		\geq \limsup_{k \to \infty} \frac{1}{f(s_k)} |X_{s_k}(\omega)-x|
		\geq \frac{1}{2} \limsup_{k \to \infty} \frac{1}{f(t_k)} \sup_{s \leq t_k} |X_s(\omega)-x| 
		\geq \frac{R}{2}.
	\end{equation*}
	As $R \geq 1$ is arbitrary, this proves \ref{main-8-i}.
	\quartolist It remains to prove \ref{main-8-ii}. To this end, we show that if $f$ is regularly varying at zero, i.e.\ 
	\begin{equation*}
		\exists \beta>0 \, \, \forall \lambda>0\::\: \lim_{t \to 0} \frac{f(\lambda t)}{f(t)}=\lambda^{\beta},
	\end{equation*}
	then \eqref{conv-eq9} for $R=1$ implies \eqref{conv-eq9} for all $R \geq 1$. The desired lower bound for the growth of the sample paths then follows from the first part of this proof. Let $C>0$ be such that \eqref{conv-eq9} holds with $R=1$. As we have seen in \primo, it follows that \eqref{conv-eq9} holds with $R=1$ for any $C>0$. Since $f$ is regularly varying at zero, there is $\lambda>0$ such that $f(\lambda t)/f(t) \geq R$ for $t>0$ small. Thus,
	\begin{align*}
		\limsup_{t \to 0} t \inf_{|z-x| \leq  R f(t)} \sup_{|\xi| \leq 1/(Cf(t))} \re q(z,\xi)
		&\geq \limsup_{t \to 0} t \inf_{|z-x| \leq  f(\lambda t)} \sup_{|\xi| \leq 1/(Cf(t))} \re q(z,\xi) \\
		&=  \frac{1}{\lambda} \limsup_{t \to 0} t \inf_{|z-x| \leq  f(t)} \sup_{|\xi| \leq 1/(Cf(t/\lambda))} \re q(z,\xi).
	\end{align*}
	Using once more that $f$ is regularly varying, we find that $f(t/\lambda) \geq \frac{1}{2} \frac{1}{\lambda^{\beta}} f(t)=:\gamma f(t)$ for $t>0$ small. Hence,
	\begin{equation*}
		\limsup_{t \to 0} t \inf_{|z-x| \leq  R f(t)} \sup_{|\xi| \leq 1/(Cf(t))} \re q(z,\xi)
		\geq \frac{1}{\lambda}  \limsup_{t \to 0} t \inf_{|z-x| \leq  f(t)} \sup_{|\xi| \leq 1/(\gamma Cf(t))} \re q(z,\xi)=\infty. \qedhere
	\end{equation*}
	
\end{proof}

The key for the proof of our final main result, Theorem~\ref{main-9}, is the following proposition.

\begin{proposition} \label{conv-3}
	Let $(X_t)_{t \geq 0}$ be a L\'evy-type process with characteristics $(b(x),0,\nu(x,dy))$ and  symbol $q$. Let $f:[0,1] \to [0,\infty)$ be a non-decreasing function. If 
	\begin{equation}
		\limsup_{n \to \infty} \frac{1}{n^2} \sup_{|z-x| \leq 3 f(1/n)} \sup_{|\xi| \leq c/f(1/n)} |q(z,\xi)|<1
		\label{conv-eq11}
	\end{equation}
	and
	\begin{equation}
		\int_0^1 \inf_{|z-x| \leq 5 f(t)} \nu(z,\{|y| >5 f(t)\}) \, dt = \infty,
		\label{conv-eq12}
	\end{equation}
	for some $c>0$ and $x \in \mbb{R}^d$, then
	\begin{equation*}
		\limsup_{t \to 0} \frac{1}{f(t)} |X_t-x| \geq 1 \quad \text{$\mbb{P}^x$-a.s.}
	\end{equation*}
\end{proposition}

\begin{remark} \label{conv-4} 
\begin{enumerate}[wide, labelwidth=!, labelindent=0pt]
	\item\label{conv-4-i} Replacing $f$ by $C\cdot f$ for $C>0$, we obtain immediately a sufficient condition for
	\begin{equation*}
		\limsup_{t \to 0} \frac{1}{f(t)} |X_t-x| \geq C \quad \text{$\mbb{P}^x$-a.s.}
	\end{equation*}
	\item\label{conv-4-ii} By the local boundedness of $q$, there is a finite constant $c=c(R,x)$ such that $|q(z,\xi)| \leq c(1+|\xi|^2)$ for all $\xi \in \mbb{R}^d$ and $|z-x| \leq R$, cf.\ \eqref{def-eq13}. Thus, $\liminf_{t \downarrow 0} f(t)/t = \infty$ is a sufficient condition for \eqref{conv-eq11}; let us mention that this growth condition on $f$ also appears in the study of upper functions for sample paths of L\'evy processes, cf.\ \cite{savov09}. More generally, if $\sup_{|z-x| \leq R} |q(z,\xi)| \leq c (1+|\xi|^{\alpha})$ for some $\alpha \in (0,2]$, then \eqref{conv-eq11} holds for any function $f$ satisfying $\liminf_{t \downarrow 0} f(t)/t^{2/\alpha}= \infty$.
\end{enumerate}
\end{remark}

\begin{proof}[Proof of Proposition~\ref{conv-3}]
	Let $x \in \mbb{R}^d$ and $c>0$ be such that \eqref{conv-eq11} and \eqref{conv-eq12} hold, and set
	\begin{equation*}
		G(x,r) := \inf_{|z-x| \leq r} \nu(z,\{|y| > r\}).
	\end{equation*} 
	Using the subadditivity of $\xi \mapsto \sqrt{|q(z,\xi)|}$, we see that \eqref{conv-eq11} actually holds for \emph{any} $c>0$. \par
	\primolist By the monotonicity of $f$ and $r \mapsto G(x,r)$, it follows from \eqref{conv-eq12} that 
	\begin{equation*}
		\sum_{n =2}^{\infty} \left(\frac{1}{n-1}-\frac{1}{n} \right) G(x,5f(1/n))
		\geq \int_0^1 G(x,5f(t)) \, dt = \infty,
	\end{equation*}
	and so 
	\begin{equation}
		\sum_{n \in \mbb{N}} \frac{1}{n^2} G(x,5f(1/n)) = \infty.
		\label{conv-eq14}
	\end{equation}
	Moreover, we note that \eqref{conv-eq12} implies that $f(t) \to 0$ as $t \downarrow 0$.
	\secundolist Denote by $(\mathcal{F}_t)_{t \geq 0}$ the canonical filtration of $(X_t)_{t \geq 0}$. We claim that
	\begin{equation}
		\sum_{n \in \mbb{N}} \mbb{E}^x(\I_{A_n} \mid \mathcal{F}_{1/(n+1)}) = \infty \quad \text{$\mbb{P}^x$-a.s.} \label{conv-eq16}
	\end{equation}
	for
	\begin{equation*}
		A_n := \bigg\{ \frac{1}{f(1/n)} \sup_{\frac{1}{n+1} \leq r < \frac{1}{n}} |X_r-x| \geq 1 \bigg\}.
	\end{equation*}
	To prove \eqref{conv-eq16} we fix $n \in \mbb{N}$ and note that, by the Markov property,
	\begin{equation*}
		\mbb{E}^x(\I_{A_n} \mid \mathcal{F}_{1/(n+1)}) = u(X_{1/(n+1)})
	\end{equation*}
	where
	\begin{equation*}
		u(z) := \mbb{P}^z \bigg( \sup_{r \leq \frac{1}{n(n+1)}} |X_r-x| \geq f(1/n) \bigg), \qquad z \in \mbb{R}^d.
	\end{equation*}
	We need a lower bound for the mapping $u$. If $z \notin B(x,f(1/n)$, then $|X_0-x| \geq f(1/n)$ $\mbb{P}^z$-a.s. which gives $u(z)=1$. Next we consider the case $z \in B(x,f(1/n))$. By the triangle inequality,
	\begin{equation*}
		u(z) \geq \mbb{P}^z \bigg( \sup_{r \leq \frac{1}{n(n+1)}} |X_r-z| \geq 2f(1/n) \bigg)=:U(z).
	\end{equation*}
	The maximal inequality \eqref{max-eq0} shows that 
	\begin{equation*}
		U(z) \leq c' \frac{1}{n(n+1)} \sup_{|z-y| \leq 2f(1/n)} \sup_{|\xi| \leq 1/(2f(1/n))} |q(y,\xi)|
	\end{equation*}
	for some absolute constant $c'>0$. Since $|z-x| \leq f(1/n)$ and $\sqrt{|q(y,\cdot)|}$ is subadditive, we get 
	\begin{equation*}
		U(z) 
		\leq 4c' \frac{1}{n (n+1)} \sup_{|y-x| \leq 3 f(1/n)} \sup_{|\xi| \leq 1/f(1/n)} |q(y,\xi)|,
	\end{equation*}
	see the proof of Corollary~\ref{up-9}. Thus, by \eqref{conv-eq11}, $U(z) \leq 1-\eps$ for $n \gg 1$ and some $\eps \in (0,1)$. Applying Corollary~\ref{max-5} and using $|z-x| \leq f(1/n)$, we find that
	\begin{equation*}
		u(z) \geq U(z) 
		\geq \eps \frac{1}{n(n+1)} G(z,4f(1/n)) 
		\geq \eps \frac{1}{n(n+1)} G(x,5f(1/n)), \qquad z \in B(x,f(1/n)),
	\end{equation*}
	for $n \gg 1$. In summary,
	\begin{equation*}
		\mbb{E}^x(\I_{A_n} \mid \mathcal{F}_{1/(n+1)})
		\geq  \min \left\{\eps \frac{1}{n(n+1)} G(x,5f(1/n)),1 \right\}
	\end{equation*}
	for $n \gg 1$. Thus, by \eqref{conv-eq14}, $\sum_{n \in \mbb{N}} \mbb{E}^x(\I_{A_n} \mid \mc{F}_{1/(n+1)}) = \infty$ $\mbb{P}^x$-a.s.
	\tertiolist The almost sure divergence of the series implies by the conditional Borel-Cantelli lemma for backward filtrations, cf.\ Proposition~\ref{appendix-1}, that
	\begin{equation*}
		\mbb{P}^x \left( \limsup_{n \to \infty} A_n \right) = 1,
	\end{equation*}
	and so there is a measurable set $\tilde{\Omega}$ with $\mbb{P}^x(\tilde{\Omega})=1$ such that
	\begin{equation*}
		\forall \omega \in \tilde{\Omega} \, \forall n \gg 1 \, \exists t_n=t_n(\omega) \in \left[ \frac{1}{n+1},\frac{1}{n} \right)\::\: \frac{1}{f(1/n)} |X_{t_n}(\omega)-x| \geq 1.
	\end{equation*}
	Using the monotonicity of $f$, we conclude that 
	\begin{equation*}
		\limsup_{t \to 0} \frac{1}{f(t)} |X_t(\omega)-x| \geq \limsup_{n \to \infty} \frac{1}{f(t_n)} |X_{t_n}(\omega)-x| \geq 1, \qquad \omega \in \tilde{\Omega}. \qedhere
	\end{equation*}
\end{proof}

\begin{proof}[Proof of Theorem~\ref{main-9}] 
	First we prove \ref{main-9-i}. Let $f\geq0$ be non-decreasing and $c>0$ such that $\int_0^1 \inf_{|z-x| \leq cf(t)} \nu(z;\{|y|>cf(t)\}) \, dt = \infty$. We consider separately the cases that \ref{C1} resp.\ \ref{C2} holds.
	\primolist Assume that \ref{C1} holds. If for every $R \geq 1$ the limit
	\begin{equation}
		\limsup_{t \to 0} t \inf_{|z-x| \leq R f(t)} \sup_{|\xi| \leq 1/(Cf(t))} \re q(z,\xi) \label{conv-eq20}
	\end{equation}
	is infinite for some constant $C=C(R)$, then Proposition~\ref{main-8} yields
	\begin{equation*}
		\limsup_{t \to 0} \frac{1}{f(t)} |X_t-x| = \infty \quad \text{$\mbb{P}^x$-a.s.}
	\end{equation*}
	On the other hand, if \eqref{conv-eq20} is finite for some $R \geq 1$ and all $C>0$, then
	\begin{align*}
		\limsup_{t \to 0} t^2 \sup_{|z-x| \leq C f(t)} \sup_{|\xi| \leq 1/f(t)} \re q(z,\xi) 
		&\leq c' \limsup_{t \to 0} t \frac{\sup_{|z-x| \leq C f(t)} \sup_{|\xi| \leq 1/f(t)} |q(z,\xi)|}{\inf_{|z-x| \leq R f(t)} \sup_{|\xi| \leq 1/f(t)} |q(z,\xi)|}
	\end{align*}
	for some constant $c'>0$, and the latter limit is zero by \ref{C1}. Hence, \eqref{conv-eq11} holds. Applying Proposition~\ref{conv-3} proves the assertion.
	\secundolist If \ref{C2} holds, then the assertion is immediate from Proposition~\ref{conv-3} and Remark~\ref{conv-4}\ref{conv-4-ii}. \par
	It remains to show \ref{main-9-ii}. To this end, assume additionally that \eqref{A1'} holds for some $R>0$ and let $f$ be non-decreasing with $\int_0^1 \inf_{|z-x| \leq c f(t)} \sup_{|\xi| \leq 1/f(t)} |q(z,\xi)| \, dt = \infty$ for some $c>0$. Then Proposition~\ref{main-7}\ref{main-7-ii} yields $\int_0^1 \inf_{|z-x| \leq cf(t)} \nu(z,\{|y|>c f(t)\}) \, dt = \infty$, and applying \ref{main-9-i} finishes the proof.
\end{proof}

\appendix
\section{}

In the proof of Proposition~\ref{conv-3} we used the following conditional Borel-Cantelli lemma for backward filtrations.

\begin{proposition} \label{appendix-1}
	Let $(\mc{F}_n)_{n \in \mbb{N}}$ be a sequence of decreasing $\sigma$-algebras. Let $(X_n)_{n \in \mbb{N}}$ be a sequence of non-negative random variables such that $X_n$ is $\mathcal{F}_n$-measurable for each $n \in \mbb{N}$. If $Y:=\sup_{n \in \mbb{N}} X_n$ is integrable, then there is a $\mbb{P}$-null set $N$ such that
	\begin{equation*}
		\left\{ \sum_{n \in \mbb{N}} \mbb{E}(X_n \mid \mathcal{F}_{n+1})= \infty \right\} \subseteq \left\{ \sum_{n \in \mbb{N}} X_n = \infty \right\} \cup N.
	\end{equation*}
\end{proposition}

The idea for our proof is from Chen \cite{chen78}.

\begin{proof}
	Set $M_n := \mbb{E}(X_n \mid \mathcal{F}_{n+1})$ and $R_n := \sum_{i=n}^{\infty} X_i$. Since $R_{n+1}$ is $\mathcal{F}_{n+1}$-measurable, we find using the tower property
	\begin{align*}
		\mbb{E} \left( \frac{1}{(1+R_1)^2} \sum_{n=1}^k M_n \right)
		= 	\mbb{E} \left( \I_{\{R_1<\infty\}} \frac{1}{(1+R_1)^2} \sum_{n=1}^k M_n \right)
		&\leq \mbb{E} \left( \sum_{n=1}^k \frac{M_n}{(1+R_{n+1})^2} \I_{\{R_{n+1}<\infty\}} \right) \\
		&= \mbb{E} \left( \sum_{n=1}^k \frac{X_n}{(1+R_{n+1})^2} \I_{\{R_{n+1}< \infty\}} \right)
	\end{align*}
	for all $k \in \mbb{N}$. Since
	\begin{align*}
		\frac{X_n}{(1+R_{n+1})^2}
		 &= \frac{(1+R_n)-(1+R_{n+1})}{(1+R_n)(1+R_{n+1})} \frac{1+R_n}{1+R_{n+1}} 
		 \leq \left( \frac{1}{1+R_{n+1}}-\frac{1}{1+R_n} \right) (1+Y)
	\end{align*}
	on $\{R_{n+1}<\infty\}$, we have
	\begin{align*}
		\sum_{n=1}^k \frac{X_n}{(1+R_{n+1})^2} 
		\leq (1+Y) \left( \frac{1}{1+R_k}-\frac{1}{1+R_1} \right)
		\leq 1+Y,
	\end{align*}
	and so \begin{equation*}
		\mbb{E} \left( \frac{1}{(1+R_1)^2} \sum_{n=1}^k M_n \right)
		\leq \mbb{E}(1+Y)<\infty.
	\end{equation*}
	Letting $k \to \infty$ using monotone convergence, we see that $\frac{1}{(1+R_1)^2} \sum_{n \in \mbb{N}} M_n < \infty$ almost surely, which proves the assertion.
\end{proof}

The following estimate was needed in the proof of Proposition~\ref{up-1}.

\begin{lemma} \label{appendix-2}
\begin{equation*}
	\sum_{n \in \mbb{N}} \frac{1}{n^2} t^{1/n} \log \frac{1}{t} \leq 2 \fa t \in (0,1).
\end{equation*}
\end{lemma}

\begin{proof}
	Fix $t \in (0,1)$. By the fundamental theorem of calculus, we have for every $n \in \mbb{N}$
	\begin{align*}
		0 
		\leq t^{1/(n+1)}-t^{1/n}
		= - \int_{1/(n+1)}^{1/n} t^r \log t \, dr
		\geq \log(t^{-1}) t^{1/n} \left( \frac{1}{n}-\frac{1}{n+1} \right) 
		= \log(t^{-1}) t^{1/n} \frac{1}{n(n+1)}.
	\end{align*}
	Thus,
	\begin{align*}
		\sum_{n \in \mbb{N}} \frac{1}{n(n+1)} t^{1/n} \log(t^{-1})
		\leq \sum_{n \in \mbb{N}} \left( t^{1/(n+1)}-t^{1/n} \right)
		= \lim_{N \to \infty} t^{1/N} - t
		=1-t.
	\end{align*}
	As $(n+1)/n \leq 2$ for all $n \in \mbb{N}$, this proves the assertion.
\end{proof}

\emph{Acknowledgment}: I'm very grateful to Ren\'e Schilling for helpful discussions and comments.


\begin{thebibliography}{99}\frenchspacing
	\bibitem{bass}
		Bass, R. F.: Uniqueness in law for pure jump Markov processes. \emph{Probab.\ Theory Rel.\ Fields} \textbf{79} (1988), 271–287.
	\bibitem{bass09} 
		Bass, R.F., Tang, H.: The martingale problem for a class of stable-like processes. \emph{Stoch.\ Proc.\ Appl.} \textbf{119} (2009), 1144--1167.
	\bibitem{bertoin}
		Bertoin, J.: \emph{L\'evy Processes}. Cambridge University Press, 1996.
	\bibitem{bertoin08} 
		Bertoin, J., Doney, R.A., Maller, R.A.: Passage of L\'evy processes across power law boundaries at small times. \emph{Ann.\ Probab.} \textbf{36} (2008), 160--197.
	\bibitem{billingsley}
		Billingsley, P.: \emph{Probability and Measure}. Wiley, 1995.
	\bibitem{bingham}
		Bingham, N.H., Goldie, C.M., Teugels, J.L.: \emph{Regular Variation}. Cambridge University Press, 1987.
	\bibitem{bg61}
		Blumenthal, R.M., Getoor, R.K.: Sample Functions of Stochastic Processes with Stationary Independent Increments. \emph{Journal of Mathematics and Mechanics} \textbf{10} (1961), 493--516.
	\bibitem{ltp}
	    B\"{o}ttcher, B., Schilling, R.\,L., Wang, J.: \emph{L\'evy-Type Processes: Construction, Approximation and Sample Path Properties}. Lecture Notes in Mathematics vol.\ \textbf{2099}, (vol.~III of the ``L\'evy Matters'' subseries). Springer, 2014.
	\bibitem{chen78}
		Chen, L. H. Y.: A Short Note on the Conditional Borel-Cantelli Lemma. \emph{Ann.\ Probab.} \textbf{6} (1978), 699--700.
	\bibitem{kim21} 
		Cho, S., Kim, P., Lee, J.: General Law of iterated logarithm for Markov processes. Preprint, arXiv 2102.01917.
	\bibitem{ethier} 
			Ethier, S.\,N., Kurtz, T.\,G.: \emph{Markov Processes: Characterization and Convergence}. Wiley, 1986.
	\bibitem{fristedt}
		Fristedt, B.: Sample functions of stochastic processes with stationary, independent increments. In: Ney, P., Port, S. (eds).:\emph{Advances in Probability}, Marcel Dekker, 1974, pp.~241--396.
	\bibitem{gikh1}
		Gikhman, I.I., Skorohod, A.V.: \emph{The Theory of Stochastic Processes I}. Springer, 1974.
	\bibitem{grz}
		Grzywny, T.: On Harnack Inequality and H\"older Regularity for Isotropic Unimodal L\'evy Processes. \emph{Pot.\ Anal.} \textbf{41} (2014), 1--29.
	\bibitem{gut} Gut, A.: \emph{Probability: A Graduate Course}. Springer, 2005.
	\bibitem{hoh}
		Hoh, W.: \emph{Pseudo-Differential Operators Generating Markov Processes}. Habilitationsschrift. Universit\"{a}t Bielefeld, 1998.
	\bibitem{ito} 
		It\^o, K.: \emph{Stochastic Processes}. Springer, 2004.
	\bibitem{jacob1} 
		Jacob, N.: \emph{Pseudo Differential Operators and Markov processes I-III}. World Scientific, 2001-2005.
	\bibitem{kallenberg}
		Kallenberg, O.: Splitting at backward times in regenerative sets. \emph{Ann.\ Probab.} \textbf{9} (1981), 781--799.
	\bibitem{khin39}
		Khintchine, A.I.: Sur la croissance locale des processus stochastiques homog\`enes \`a accoroissements independants. (Russian, French summary) \emph{Izvestia Akad.\ Nauk SSSR Ser.\ Math.} \textbf{3} (1939), 487--508.
	\bibitem{barca}  
		Khoshnevisan, D., Schilling, R.L.: \emph{From Lévy-type Processes to Parabolic SPDES}. Birkh\"auser, 2016.
	\bibitem{kim17} 
		Kim, P., Kumagai, T., Wang, J.: Laws of the iterated logarithm for symmetric jump processes. \emph{Bernoulli} \textbf{23} (2017), 2330--2379.
	\bibitem{kol} 
		Kolokoltsov, V. N.: \emph{Markov processes, semigroups and generators}. De Gruyter, 2011.
	\bibitem{knop13} 
		Knopova, V., Schilling, R.L.: A note on the existence of transition probability densities of L\'evy processes. \emph{Forum Math.} \textbf{25} (2013), 125--149.
	\bibitem{knop14} 
		Knopova, V., Schilling, R.L.: On the small-time behaviour of Lévy-type processes. \emph{Stoch.\ Proc.\ Appl.} \textbf{124} (2014), 2249--2265.
	\bibitem{matters} 
		K\"uhn, F.: \emph{L\'evy-Type Processes: Moments, Construction and Heat Kernel Estimates}. Lecture Notes in Mathematics vol.\ \textbf{2187}, (vol.~VI of the ``L\'evy Matters'' subseries). Springer, 2017.
	\bibitem{sde} 
		K\"uhn, F.: Solutions of L\'evy-driven SDEs with unbounded coefficients as Feller processes \emph{Proc.\ Amer.\ Math.\ Soc.} \textbf{146} (2018), 3591--3604.
	\bibitem{mp-feller} 
		K\"uhn, F.: On martingale problems and Feller processes. \emph{Electron.\ J.\ Probab.} \textbf{23} (2018), 1--18.
	\bibitem{markovian} 
		K\"uhn, F.: Existence of (Markovian) solutions to martingale problems associated with L\'evy-type operators. \emph{Electron. J. Probab.} \textbf{25} (2020), 1--26.
	\bibitem{ihke} 
		K\"uhn, F., Schilling, R.L.: On the domain of fractional Laplacians and related generators of Feller processes. \emph{J.\ Funct.\ Anal.} \textbf{276} (2019), 2397--2439.
	\bibitem{kurtz} 
		Kurtz, T.G.: Equivalence of Stochastic Equations and Martingale Problems. In: D. Crisan (Ed.), \emph{Stochastic Analysis 2010}. Springer, 2011, pp. 113--130.
	\bibitem{negoro94}
		Negoro, A.: Stable-like processes: construction of the transition density and the behavior of sample paths near $t=0$. \emph{Osaka J. Math.} \textbf{31} (1994), 189--214.
	\bibitem{pruitt81} 
		Pruitt, W. E.: The Growth of Random Walks and L\'evy Processes. \emph{Ann.\ Probab.} \textbf{9} (1981), 948--956.
	\bibitem{reker20}
		Reker, J.: Short-time behavior of solutions to L\'evy-driven SDEs. Preprint, arXiv 2008.00526.
	\bibitem{sato}
		Sato, K.: \emph{L\'evy processes and infinitely divisible distributions}. Cambridge University Press, Cambridge 2013 (revised ed).
	\bibitem{savov09}
	  	Savov, M.: Small time two-sided LIL behavior for Lévy processes at zero. \emph{Probab. Theory Related Fields} \textbf{144} (2009), 79--98.
	\bibitem{rs-growth} 
		Schilling, R. L.: Growth and H\"older conditions for the sample paths of Feller processes. \emph{Probab.\ Theory Related Fields}, \textbf{112} (1998), 565--611.
	\bibitem{schnurr10} 
		Schilling, R., Schnurr, A.: The Symbol Associated with the Solution of a Stochastic Differential Equation \emph{Electron.\ J.\ Probab.} \textbf{15} (2010), 1369--1393.
	 \bibitem{skorohod91} 
	 	Skorokhod, A. V.: \emph{Random processes with independent increment}. Springer, 1991.
	\bibitem{taylor67} 
		Taylor, S. J.: Sample Path Properties of a Transient Stable Process. \emph{J.\ Math.\ Mech.} \textbf{16} (1967), 1229–1246.
	\bibitem{wee88}
		Wee, I.S., Kim, Y.K.: General laws of the iterated logarithm for L\'evy processes. \emph{J.\ Kor.\ Statist.\ Soc.} \textbf{17} (1988), 30--45.
\end{thebibliography}
\end{document}